\documentclass[twoside, preprint]{imsart}
\usepackage[sort, round]{natbib}
\usepackage[english]{babel}
\usepackage{parskip}
\usepackage{amsmath,amsfonts,amsthm, amssymb}
\usepackage{graphicx}
\usepackage{subfigure}
\usepackage{enumerate}

\RequirePackage[colorlinks,citecolor=blue,urlcolor=blue]{hyperref}

\startlocaldefs
\newcommand{\E}{\mathbb{E}}
\newcommand{\1}{\textbf{1}}
\newcommand{\zt}{\zeta_\tau}
\newcommand{\var}{{\rm Var}}

\theoremstyle{plain}
\newtheorem{thm}{Theorem}[section]
\newtheorem{lem}[thm]{Lemma}

\graphicspath{{./FinalFiguresPaper/}}

\endlocaldefs

\begin{document}

\begin{frontmatter}

\title{The Horseshoe Estimator: Posterior Concentration around Nearly Black Vectors} 
\runtitle{The horseshoe estimator: posterior concentration}

\begin{aug}
\author{\fnms{S.L.} \snm{van der Pas}\corref{}\thanksref{t1} \ead[label=e1]{svdpas@math.leidenuniv.nl}}
\address{Mathematical Institute, Leiden University \\ \printead{e1}}
\thankstext{t1}{Research supported by Netherlands Organization for Scientific
Research NWO.}
\author{\fnms{B.J.K.} \snm{Kleijn}\ead[label=e3]{b.kleijn@uva.nl}}
\address{Korteweg-de Vries Institute for Mathematics, University of Amsterdam \\ \printead{e3}}
\and \author{\fnms{A.W.} \snm{van der Vaart\thanksref{t2}}\ead[label=e2]{avdvaart@math.leidenuniv.nl}}
\address{Mathematical Institute, Leiden University \\ \printead{e2}}
\thankstext{t2}{The research leading to these results has received funding from the
European
  Research Council under ERC Grant Agreement 320637.}

\runauthor{S.L. van der Pas, B.J.K. Kleijn and A.W. van der Vaart}
\end{aug}

\begin{abstract}
We consider the horseshoe estimator due to \cite{Carvalho2010} for the multivariate normal mean model in the situation that the mean vector is sparse in the nearly black sense. We assume the frequentist framework where the data is generated according to a fixed mean vector. We show that if the number of nonzero parameters of the mean vector is known, the horseshoe estimator attains the minimax $\ell_2$ risk, possibly up to a multiplicative constant.  We provide conditions under which the horseshoe estimator combined with an empirical Bayes estimate of the number of nonzero means still yields the minimax risk. We furthermore prove an upper bound on the rate of contraction of the posterior distribution around the horseshoe estimator, and a lower bound on the posterior variance. These bounds indicate that the posterior distribution of the horseshoe prior may be more informative than that of other one-component priors, including the Lasso. 
\end{abstract}





\end{frontmatter}

\section{Introduction}
We consider the normal means problem, where we observe a vector $Y \in \mathbb{R}^n$, $Y = (Y_1, \ldots, Y_n)$,  such that
\begin{equation*}
Y_i = \theta_i + \varepsilon_i, \quad i = 1, \ldots, n,
\end{equation*}
for independent normal random variables $\varepsilon_i$ with mean zero and variance $\sigma^2$. The vector  $\theta = (\theta_1, \ldots, \theta_n)$ is assumed to be sparse, in the `nearly black' sense that the number of nonzero means 
\begin{equation*}
p_n := \#\{i : \theta_i \neq 0 \} 
\end{equation*} 
is  $o(n)$ as $n \to \infty$. A natural Bayesian approach to recovering $\theta$ would be to induce sparsity through a `spike and slab' prior \citep{Mitchell1988}, which consists of a mixture of a Dirac measure at zero and a (heavy-tailed) continuous distribution. \cite{Johnstone2004} analyzed an empirical Bayes version of this approach, where the mixing weight is obtained by marginal maximum likelihood. In the frequentist set-up that the data is generated according to a fixed mean vector, they showed that the empirical Bayes coordinatewise posterior median attains the minimax rate, in $\ell_q$ norm, $q \in (0, 2]$, for mean vectors that are either nearly black or of bounded $\ell_p$ norm, $p \in (0, 2]$. \cite{Castillo2012} analyzed a fully Bayesian version, where the proportion of nonzero coefficients is modelled by a prior distribution. They identified combinations of priors on this proportion and on the nonzero coefficients (the `slab') that yield posterior distributions concentrating around the underlying mean vector at the minimax rate in $\ell_q$ norm, $q \in (0, 2]$, for mean vectors that are nearly black, and in $\ell_q$ norm, $q \in (0,2)$  for mean vectors of bounded weak $\ell_p$ norm, $p \in (0,q)$. Other work on empirical Bayes approaches to the two-group model includes \citep{Efron2008, Jiang2009, Yuan2005}. 

As a full Bayesian approach with a mixture of a Dirac and a continuous component may require exploration of a model space of size $2^n$, implementation on large datasets is currently impractical,  although \cite{Castillo2012} present an algorithm which can compute several aspects of the posterior in polynomial time, provided sufficient memory can be allocated. Several authors, including  \citep{Armagan2013, Griffin2010}, have proposed one-component priors, which model the spike at zero by a peak in the prior density at this point. For most of these proposals, theoretical justification in terms of minimax risk rates or posterior contraction rates is lacking. The Lasso estimator \citep{Tibshirani1996}, which arises as the MAP estimator after placing a Laplace prior with common parameter on each $\theta_i$, is an exception. It attains close to the minimax risk rate in $\ell_q$, $q \in [1, 2]$ (\cite{Bickel2009}). It has however been recently shown that the corresponding full posterior distribution contracts at a much slower rate than the mode \citep{Castillo2013}. This is undesirable, because this implies that the posterior distribution cannot provide an adequate measure of uncertainty in the estimate.  
 
In general one would use a posterior distribution both for recovery and for uncertainty quantification. For the first, a measure of centre, such as a median or mode, suffices. For the second, one typically employs a credible set, which is defined as a central set of prescribed posterior probability. For realistic uncertainty quantification it is necessary that the posterior contracts to its center at the same rate as the posterior median or mode approaches the true parameter.   
 
In this paper we study the posterior distribution resulting from the horseshoe prior, which is a one-component prior, introduced in \citep{Carvalho2010, Carvalho2009} and expanded upon in \citep{Scott2011, Polson2012, Polson2012-2}. It combines a pole at zero with Cauchy-like tails. The corresponding estimator does not face the computational issues  of the point mass mixture models. \cite{Carvalho2010} already showed good behaviour of the horseshoe estimator in terms of Kullback-Leibler risk when the true mean is zero. \cite{Datta2013} proved some optimality properties of a multiple testing rule induced by the horseshoe estimator.  In this paper, we prove that the horseshoe estimator achieves the minimax quadratic risk, possibly up to a multiplicative constant. We furthermore prove that the posterior variance is of the order of the minimax risk, and thus the posterior contracts at the minimax rate around the underlying mean vector. These results are proven under the assumption that the number $p_n$ of nonzero parameters is known. However, we also provide conditions under which the horseshoe estimator combined with an empirical Bayes estimator still attains the minimax rate, when $p_n$ is unknown.

This paper is organized as follows. In section \ref{sec:horseshoe}, the horseshoe prior is described and a summary of simulation results is given. The main results, that the horseshoe estimator attains the minimax squared error risk (up to a multiplicative constant) and that the posterior distribution contracts around the truth at the minimax rate, are stated in section \ref{sec:results}. Conditions on an empirical Bayes estimator of the key parameter $\tau$ such that the minimax $\ell_2$ risk  will still be obtained are given in section \ref{sec:empiricalbayes}. The behaviour of such an empirical Bayes estimate is compared to a full Bayesian version in a numerical study in section \ref{sec:simulation}. Section \ref{sec:finalremarks} contains some concluding remarks. The proofs of the main results and supporting lemmas are in the appendix.

\subsection{Notation}
We write $A_n \asymp B_n$ to denote $0 < \lim_{n \to \infty} \inf \frac{A_n}{B_n} \leq \lim_{n \to \infty} \sup \frac{A_n}{B_n} < \infty$ and $A_n \lesssim B_n$ to denote that there exists a positive constant $c$ independent of $n$ such that $A_n \leq cB_n$. $A \vee B$ is the maximum of $A$ and $B$, and $A \wedge B$ the minimum of $A$ and $B$. The standard normal density and cumulative distribution are denoted by $\phi$ and $\Phi$ and we set $\Phi^c = 1 - \Phi$. The norm $\| \cdot \|$ will be the $\ell_2$ norm and the class of nearly black vectors will be denoted by $\ell_0[p_n] := \{\theta \in \mathbb{R}^n : \#(1 \leq i \leq n : \theta_i \neq 0) \leq p_n \}$.

\section{\label{sec:horseshoe}The horseshoe prior}
In this section, we give an overview of some known properties of the horseshoe estimator which will be relevant to the remainder of our discussion. The horseshoe prior for a parameter $\theta$ modelling an observation $Y \sim \mathcal{N}(\theta, \sigma^2I_n)$ is defined hierarchically \citep{Carvalho2010}:
\begin{equation*}
\theta_i \mid \lambda_i, \tau \sim \mathcal{N}(0, \sigma^2\tau^2\lambda_i^2), \quad \lambda_i \sim C^+(0, 1),
\end{equation*} 
for $i = 1, \ldots, n$, where $C^+(0,1)$ is a standard half-Cauchy distribution. The parameter $\tau$ is assumed to be fixed in this paper, rendering the $\theta_i$ independent \textit{a priori}. The corresponding density $p_\tau$ increases logarithmically around zero, while its tails decay quadratically. The posterior density of $\theta_i$ given $\lambda_i$ and $\tau$ is normal with mean  $(1-\kappa_i)y_i$, where $\kappa_i = \frac{1}{1 + \tau^2\lambda_i^2}$. Hence, by Fubini's theorem: 
\begin{equation*}
\E[\theta_i \mid y_i, \tau] = (1 - \E[\kappa_i \mid y_i, \tau])y_i. 
\end{equation*} 
The posterior mean $\E[\theta \mid y, \tau]$ will be referred to as the horseshoe estimator and denoted by $T_\tau(y)$. The horseshoe prior takes its name from the prior on $\kappa_i$, which is given by:
\begin{equation*}
p_\tau(\kappa_i) = \frac{\tau}{\pi}\frac{1}{1 - (1-\tau^2)\kappa_i}(1-\kappa_i)^{-\frac{1}{2}}\kappa_i^{-\frac{1}{2}}.
\end{equation*}
If $\tau = 1$, this reduces to a Be$(\tfrac{1}{2}, \tfrac{1}{2})$ distribution, which looks like a horseshoe. As illustrated in Figure \ref{fig:effecttauonprior}, decreasing $\tau$ skews the prior distribution on $\kappa_i$ towards one, corresponding to more mass near zero in the prior on $\theta_i$ and a stronger shrinkage effect in $T_\tau(y)$.

\begin{figure}
\begin{center}
\subfigure
{
\label{fig:prioronkappa}
\includegraphics[scale=0.21]{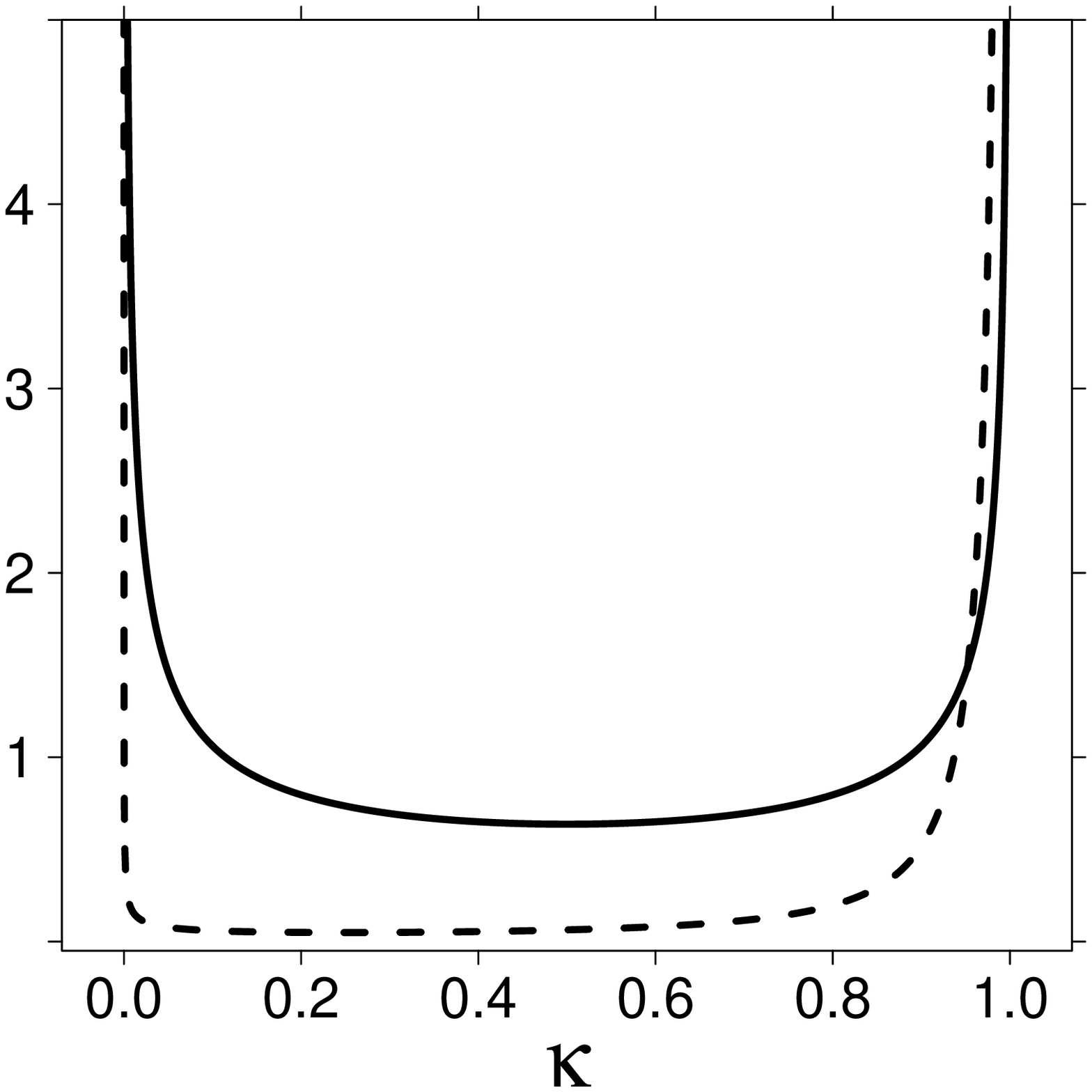}
}
\subfigure
{
\label{fig:priorontheta}
\includegraphics[scale=0.21]{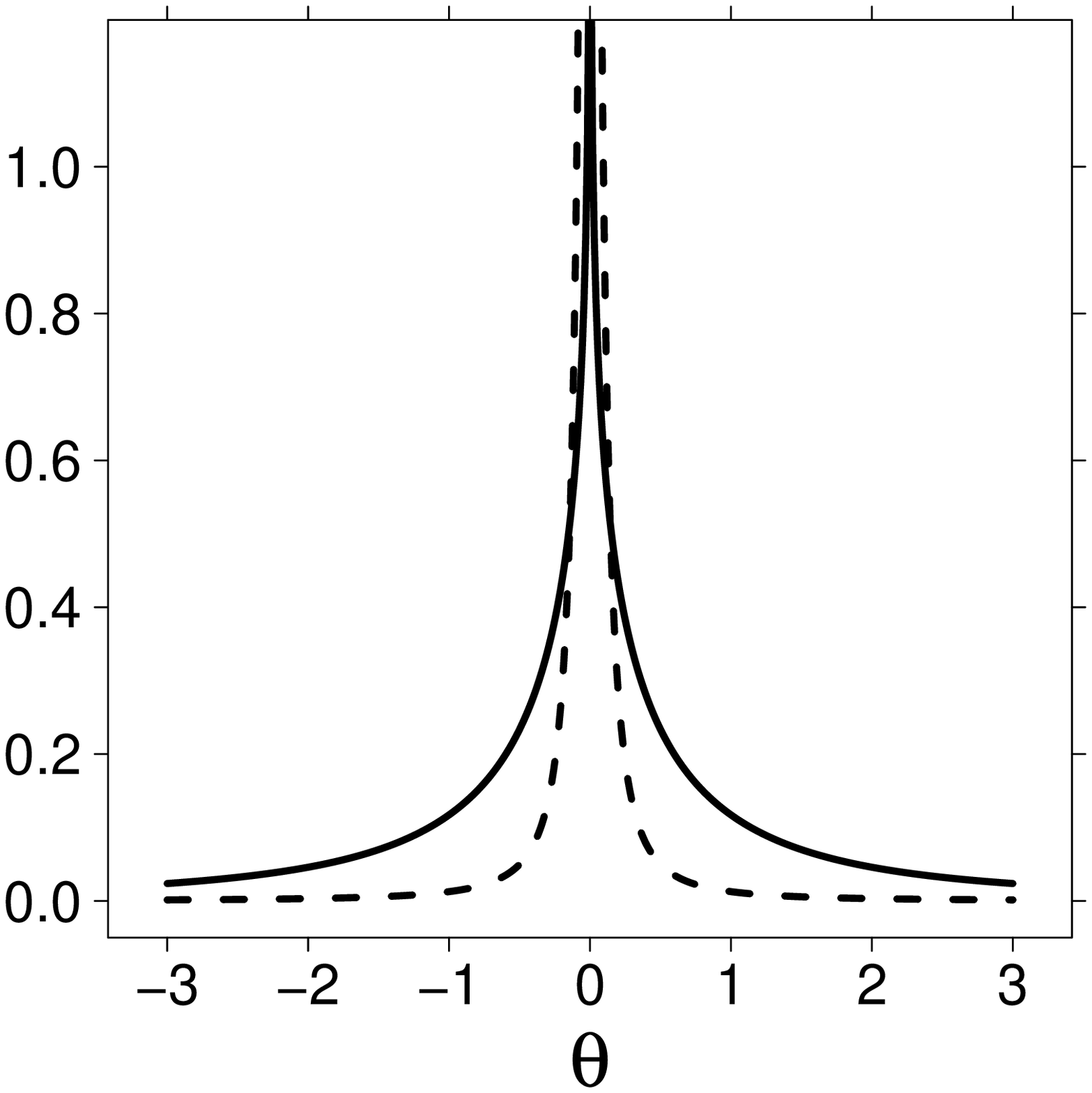}
}
\subfigure
{
\label{fig:effecttauonpm}
\includegraphics[scale=0.21]{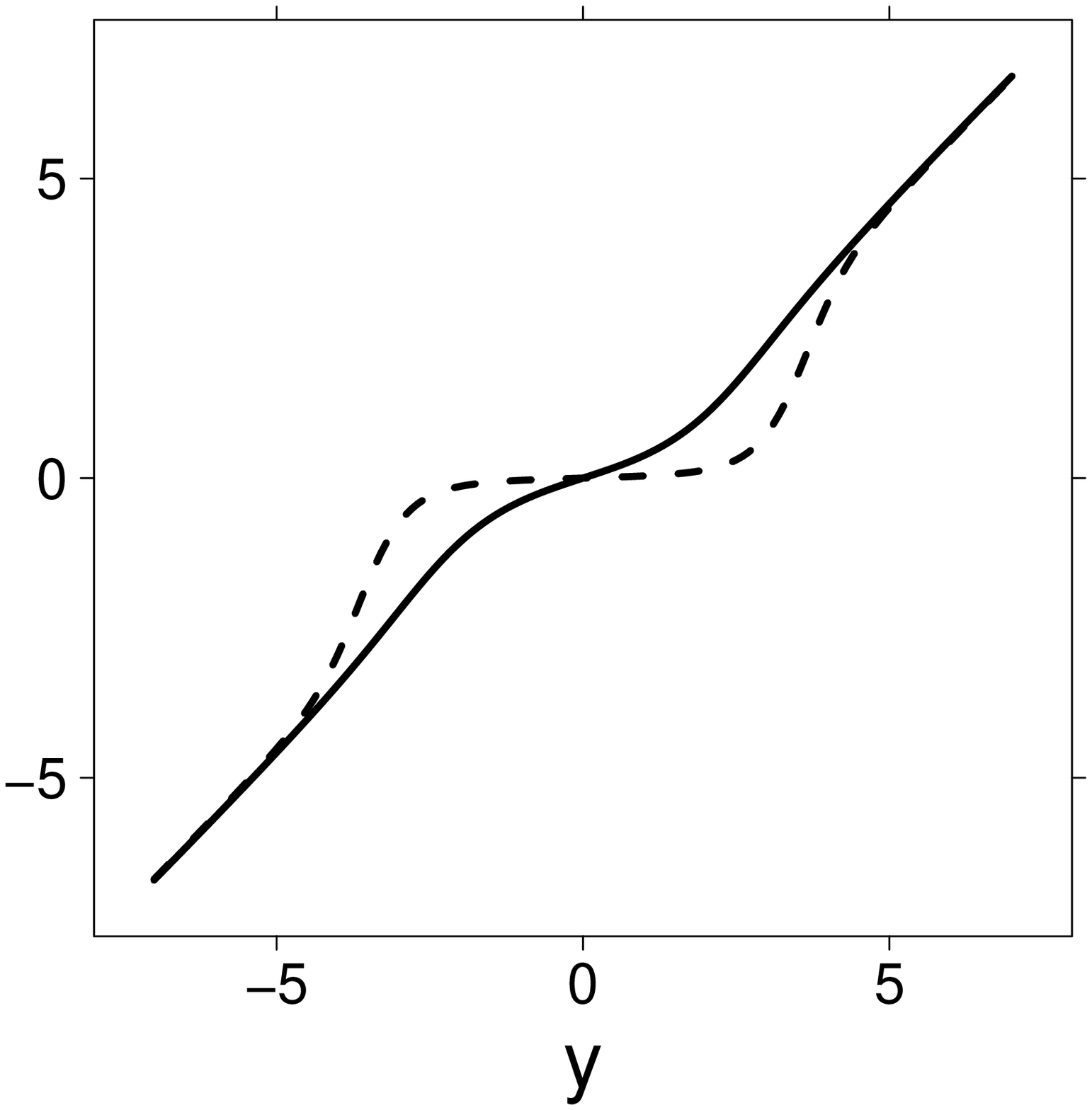}
}
\caption{The effect of decreasing $\tau$ on the priors on $\kappa$ (left) and $\theta$ (middle) and the posterior mean $T_\tau(y)$ (right). The solid line corresponds to $\tau = 1$, the dashed line to $\tau = 0.05$. Decreasing $\tau$ results in a higher prior probability of shrinking the observations towards zero. }
\label{fig:effecttauonprior}
\end{center}
\end{figure}

The posterior mean can be expressed as:
\begin{equation}\label{eq:pmexpression}
T_\tau(y_i)
= y_i\left(1 - \frac{2 \Phi_1\left(\tfrac{1}{2}, 1, \tfrac{5}{2};  \tfrac{y_i^2}{2\sigma^2}, 1 - \tfrac{1}{\tau^2} \right) }{3 \Phi_1\left(\tfrac{1}{2}, 1, \tfrac{3}{2};  \tfrac{y_i^2}{2\sigma^2}, 1 - \tfrac{1}{\tau^2} \right)  } \right)  = y_i\frac{ \int_0^1 z^{\frac{1}{2}}\frac{1}{\tau^2 + \left(1- \tau^2 \right)z }e^{\frac{y_i^2}{2\sigma^2}z} dz  } {\int_0^1 z^{-\frac{1}{2}}\frac{1}{\tau^2 + \left(1 - \tau^2 \right)z }e^{\frac{y_i^2}{2\sigma^2}z} dz },  
\end{equation}
where $\Phi_1(\alpha, \beta, \gamma; x, y)$ denotes the degenerate hypergeometric function of two variables \citep{Gradshteyn1965}. 

An unanswered question so far has been how $\tau$ should be chosen. Intuitively, $\tau$ should be small if the mean vector is very sparse, as the horseshoe prior will then place more of its mass near zero. By approximating the posterior distribution of $\tau^2$ given $\kappa = (\kappa_1, \ldots, \kappa_n)$ in case a prior on $\tau$ is used, \cite{Carvalho2010} show that if most observations are shrunk near zero, $\tau$ will be very small with high probability. They suggest a half-Cauchy prior on $\tau$. \cite{Datta2013} implemented this prior on $\tau$ and their plots of posterior draws for $\tau$ at various sparsity levels indicate the expected relationship between $\tau$ and the sparsity level: the posterior distribution of $\tau$ tends to concentrate around smaller values when the underlying mean vector is sparser. As will be discussed further in the next section, the value $\tau = \frac{p_n}{n}$ (up to a log factor) is optimal in terms of mean square error and posterior contraction rates.

In case $\tau$ is estimated empirically, as will be considered in section \ref{sec:empiricalbayes}, the horseshoe estimator can be computed by plugging this estimate into expression (\ref{eq:pmexpression}), thereby avoiding the use of MCMC. Other aspects of the posterior, such as the posterior variance, can be computed using such a plug-in procedure as well. \cite{Polson2012} and \cite{Polson2012-2} consider computation of the horseshoe estimator based on the representation in terms of degenerate hypergeometric functions, as these can be efficiently computed using  converging series of confluent hypergeometric functions. They report unproblematic computations for $\tau^2$ between $\frac{1}{1000}$ and $1000$. A second option is to apply a quadrature routine to the integral representation in (\ref{eq:pmexpression}). As the continuity and symmetry of $T_\tau(y)$ in $y$ can be taken advantage of when computing the horseshoe estimator for a large number of observations, the complexity of these computations mostly depends on the value of $\tau$. Both approaches will be slower for smaller values of $\tau$. Hence, if we use the (estimated) sparsity level $\frac{p_n}{n}$ (up to a log factor) for $\tau$, the computation of the horseshoe estimator will be slower if there are fewer nonzero parameters. As noted by \cite{Scott2010-2}, problems arise in Gibbs sampling precisely when $\tau$ is small as well. Hence care needs to be taken with any computational approach if $\frac{p_n}{n}$ is suspected to be very close to zero. 

The performance of the horseshoe prior, with additional priors on $\tau$ and $\sigma^2$, in various simulation studies has been very promising. \cite{Carvalho2010} simulated sparse data where the nonzero components were drawn from a Student-$t$ density and found that the horseshoe estimator systematically beat the MLE, the double-exponential (DE) and normal-exponential-gamma (NEG) priors, and the empirical Bayes model due to \cite{Johnstone2004} in terms of square error loss. Only when the signal was neither sparse nor heavy-tailed did the MLE, DE and NEG priors have an edge over the horseshoe estimator. In  similar experiments in \citep{Carvalho2009, Polson2012} the horseshoe prior outperformed the DE prior, while behaving similarly to a heavy-tailed discrete mixture.  In a wavelet-denoising experiment under several noise levels and loss functions, the horseshoe estimator compared favorably to the discrete wavelet transform and the empirical Bayes model \citep{Polson2010}.  \cite{Bhattacharya2012} applied several shrinkage priors to data with the underlying mean vector consisting of zeroes and fixed nonzero values and found the posterior median of the horseshoe prior performing better in terms of squared error than the Bayesian Lasso (BL), the Lasso, the posterior median of a point mass mixture prior as in \citep{Castillo2012} and the empirical Bayes model proposed by \cite{Johnstone2004}, and comparable to their proposed Dirichlet-Laplace (DL) prior with parameter $\frac{1}{n}$. Results in \citep{Armagan2013} are similar. In a second simulation setting, \cite{Bhattacharya2012} generated data of length $n = 1000$, with the first ten means equal to 10, the next 90 equal to a number $A \in \{2, \ldots, 7\}$ and the remainder equal to zero. In this simulation, the horseshoe prior beat the BL (except when $A = 2$) and the DL prior with parameter $\frac{1}{n}$ (except when $A = 7$), while performing similarly to the DL prior with parameter $\frac{1}{2}$. It is worthy of note that \cite{Koenker2014} generated data according to the same scheme and applied the empirical Bayes procedures due to \cite{Martin2013} (EBMW) and \cite{Koenker2013} (EBKM) to it. The MSE of EBMW was lower than that of the horseshoe prior for $A \in \{5, 6, 7\}$, while that of EBKM was much lower in all cases.

\section{\label{sec:results}Mean square error and bounds on the posterior variance}

In this section, we study the mean square error of the horseshoe estimator, and the spread of the posterior distribution, under the assumption that the number of nonzero parameters $p_n$ is known. Theorem \ref{thm:mse} provides an upper bound on the mean square error, and shows that for a range of choices of the global parameter $\tau$, the horseshoe estimator  attains the minimax $\ell_2$ risk, possibly up to a multiplicative constant. Theorem \ref{thm:posteriorcontraction} states upper bounds on the rate of contraction of the posterior distribution around the underlying mean vector and around the horseshoe estimator, again for a range of values of $\tau$. These upper bounds are equal, up to a multiplicative constant, to the minimax risk. The contraction rate around the truth is sharp, but this may not be the case for the rate of contraction around the horseshoe estimator. Theorems \ref{thm:varlowerbound} and \ref{thm:mismatch} provide more insight into the spread of the posterior distribution for various values of $\tau$ and indicate that $\tau = \frac{p_n}{n}\sqrt{\log(n/p_n)}$ is a good choice.\\

\begin{thm} \label{thm:mse}
Suppose $Y \sim \mathcal{N}(\theta_0, \sigma^2 I_n)$. Then the estimator $T_\tau(y)$ satisfies 
\begin{equation}\label{eq:mseupperbound}
\sup_{\theta_0 \in \ell_0[p_n]} \E_{\theta_0} \|T_\tau(Y) - \theta_0\|^2 \lesssim p_n\log\frac{1}{\tau} + (n-p_n)\tau\sqrt{\log\frac{1}{\tau}}
\end{equation}
for $\tau \to 0$, as $n, p_n \to \infty$ and $p_n = o(n)$\\
\end{thm}

By the minimax risk result in \citep{Donoho1992}, we also have a lower bound:
\begin{equation*}
\sup_{\theta_0 \in \ell_0[p_n]} \E_{\theta_0} \|T_\tau(Y)-\theta_0\|^2 \geq 2\sigma^2p_n\log\frac{n}{p_n}(1+o(1)),
\end{equation*}
as $n, p_n \to \infty$ and $p_n = o(n)$. The choice $\tau = \left(\frac{p_n}{n}\right)^\alpha$, for $\alpha \geq 1$, leads to an upper bound \eqref{eq:mseupperbound} of order $p_n\log(n/p_n)$, with (as can be seen from the proof) a multiplicative constant of at most $4\alpha\sigma^2$. Thus, for this choice of $\tau$, we have:
\begin{equation*}
\sup_{\theta_0 \in \ell_0[p_n]} \E_{\theta_0} \|T_\tau(Y)-\theta_0\|^2 \asymp p_n\log\frac{n}{p_n}.
\end{equation*}
The horseshoe estimator therefore performs well as a point estimator, as it attains the minimax risk (possibly up to a multiplicative constant of at most 2 for $\alpha = 1$). This may seem surprising, as the prior does not include a point mass at zero to account for the assumed sparsity in the underlying mean vector. Theorem \ref{thm:mse} shows that the pole at zero of the horseshoe prior mimics the point mass well enough, while the heavy tails ensure that large observations are not shrunk too much.

An upper bound on the rate of contraction of the posterior can be obtained through an upper bound on the posterior variance. The posterior variance can be expressed as:
\begin{equation*}
\var(\theta_i \mid y_i) = \frac{\sigma^2}{y_i}T_\tau(y_i) - \left(T_\tau(y_i) - y_i\right)^2 
 + y_i^2 \frac{8\Phi_1\left(\tfrac{1}{2}, 1, \tfrac{7}{2};  \tfrac{y_i^2}{2\sigma^2}, 1 - \tfrac{1}{\tau^2} \right)} {15\Phi_1\left(\tfrac{1}{2}, 1, \tfrac{3}{2};  \tfrac{y_i^2}{2\sigma^2}, 1 - \tfrac{1}{\tau^2} \right)}.
\end{equation*}
Details on the computation can be found in Lemma \ref{lem:postvar}. Using a similar approach as when bounding the $\ell_2$ risk, we can find an upper bound on the expected value of the posterior variance. \\

\begin{thm}\label{thm:var}
Suppose $Y \sim \mathcal{N}(\theta_0, \sigma^2 I_n)$. Then the variance of the posterior distribution corresponding to the horseshoe prior satisfies
\begin{equation}\label{eq:varupperbound}
\sup_{\theta_0 \in \ell_0[p_n]} \E_{\theta_0}\sum_{i=1}^{n} \var(\theta_{0i} \mid Y_{i}) \lesssim p_n\log\frac{1}{\tau} +(n-p_n)\tau\sqrt{\log\frac{1}{\tau}}
\end{equation}
for $\tau \to 0$, as $n, p_n \to \infty$ and $p_n = o(n)$.\\
\end{thm}
Again, the choice $\tau = \left(\frac{p_n}{n}\right)^\alpha$, for $\alpha \geq 1$ leads to an upper bound \eqref{eq:varupperbound} of the order of the minimax risk. This result indicates that the posterior contracts fast enough to be able to provide a measure of uncertainty of adequate size around the point estimate. Theorems \ref{thm:mse} and \ref{thm:var} combined allow us to find an upper bound on the rate of contraction of the full posterior distribution, both around the underlying mean vector and around the horseshoe estimator.\\

\begin{thm}\label{thm:posteriorcontraction}
Under the assumptions of Theorem \ref{thm:mse}, with $\tau = \left(\frac{p_n}{n}\right)^\alpha$, $\alpha \geq 1$:
\begin{equation}\label{eq:contractiontruth}
\sup_{\theta_0 \in \ell_0[p_n]} \E_{\theta_0} \Pi_\tau\left(\left. \theta : \|\theta - \theta_0\|^2 >  M_np_n\log\frac{n}{p_n} \ \right| \ Y \right) \to 0,
\end{equation}
and
\begin{equation}\label{eq:contractionestimator}
\sup_{\theta_0 \in \ell_0[p_n]} \E_{\theta_0} \Pi_\tau\left(\left. \theta : \|\theta - T_\tau(Y)\|^2 >  M_np_n\log\frac{n}{p_n} \ \right| \ Y \right) \to 0,
\end{equation}
for every $M_n \to \infty$ as $n \to \infty$.
\end{thm} 

\begin{proof}
Combine Markov's inequality with the results of Theorems \ref{thm:mse} and \ref{thm:var} for (\ref{eq:contractiontruth}), and only with the result of Theorem \ref{thm:var} for (\ref{eq:contractionestimator}).
\end{proof}

 A remarkable aspect of the preceding Theorems is that many choices of $\tau$, such as $\tau = \left(\frac{p_n}{n}\right)^\alpha$ for any $\alpha \geq 1$,  lead to an upper bound of the order $p_n\log (n/p_n)$ on the worst case $\ell_2$ risk and posterior contraction rate. The upper bound on the rate of contraction in (\ref{eq:contractiontruth}) is sharp, as the posterior cannot contract faster than the minimax rate around the true mean vector \citep{Ghosal2000}. However, this is not necessarily the case for the upper bound in (\ref{eq:contractionestimator}), and for $\tau = \left(\frac{p_n}{n}\right)^\alpha$ with $\alpha > 1$, the posterior spread may be of smaller order than the rate at which the horseshoe estimator approaches the underlying mean vector. Theorems \ref{thm:varlowerbound} and \ref{thm:mismatch} provide more insight into the effect of choosing different values of $\tau$ on the posterior spread and mean square error.\\

\begin{thm}\label{thm:varlowerbound}
Suppose $Y \sim \mathcal{N}(\theta_0, \sigma^2 I_n)$, $\theta_0 \in \ell_0[p_n]$. Then the variance of the posterior distribution corresponding to the horseshoe prior satisfies 
\begin{equation}\label{eq:varlowerbound}
\inf_{\theta_0 \in \ell_0[p_n]}\E_{\theta_0}\sum_{i=1}^{n} \var(\theta_{0i} \mid Y_{i}) \gtrsim (n-p_n)\tau\sqrt{\log\frac{1}{\tau}}
\end{equation}
for $\tau \to 0$ and $p_n = o(n)$, as $n \to \infty$. This lower bound is sharp for vectors $\theta_{0,n}$ with $p_n$ entries equal to $a_n$ and the remaining entries equal to zero, if $a_n$ is such that $|a_n| \lesssim 1/\sqrt{\log(1/\tau)}$. \\
\end{thm} 

\begin{thm}\label{thm:mismatch}
Suppose $Y \sim \mathcal{N}(\theta_{0,n}, \sigma^2 I_n)$ and $\theta_{0,n} \in \ell_0[p_n]$ is such that $p_n$ entries are equal to $\gamma\sqrt{2\sigma^2\log(1/\tau)}$, $\gamma \in (0,1)$, and all remaining entries are equal to zero. Then:
\begin{equation}\label{eq:mismatchbias}
\E_{\theta_{0,n}} \|T_\tau(Y) - \theta_{0,n} \|^2 \asymp p_n\log\frac{1}{\tau} + (n-p_n)\tau\sqrt{\log\frac{1}{\tau}} ,
\end{equation}
and
\begin{equation}\label{eq:mismatchvar}
\E_{\theta_{0,n}} \sum_{i=1}^n \var(\theta_{0,ni} \mid Y_i) \asymp p_n\tau^{(1-\gamma)^2}\left(\log \frac{1}{\tau}\right)^{\gamma - \tfrac{1}{2}} + (n-p_n)\tau\sqrt{\log\frac{1}{\tau}} ,
\end{equation}
for $\tau \to 0$ and $p_n = o(n)$, as $n \to \infty$. \\
\end{thm}

Consider $\tau = \left(\frac{p_n}{n}\right)^\alpha$. Three cases can be discerned:
\begin{enumerate}[(i)]
\item $0 < \alpha < 1$. Lower bound \eqref{eq:varlowerbound} may exceed the minimax rate, implying suboptimal spread of the posterior distribution in the squared $\ell_2$ sense.
\item $\alpha = 1$. Bounds \eqref{eq:varupperbound} and \eqref{eq:varlowerbound} differ by a factor $\sqrt{\log(n/p_n)}$, as do \eqref{eq:mismatchbias} and \eqref{eq:mismatchvar}. The gap can be closed by choosing $\tau = \frac{p_n}{n}\sqrt{\log\frac{n}{p_n}}$.
\item $\alpha > 1$. Bound \eqref{eq:varlowerbound} is not very informative, but Theorem \ref{thm:mismatch} exhibits a sequence $\theta_{0,n} \in \ell_0[p_n]$ for which there is a mismatch between the order of the mean square error and the posterior variance. Bounds \eqref{eq:mismatchbias} and \eqref{eq:mismatchvar} are of the orders $p_n(\log(1/\tau) + \tau^{1-1/\alpha}\sqrt{\log(1/\tau)})$ and $p_n( \tau^{(1-\gamma)^2}(\log(1/\tau))^{\gamma - 1/2} + \tau^{1-1/\alpha}\sqrt{\log(1/\tau)} )$, respectively. Hence up to logarithmic factors the total posterior variance \eqref{eq:mismatchvar} is a factor $\tau^{(1-1/\alpha) \wedge (1-\gamma)^2}$ smaller than the square distance of the center of the posterior to the truth \eqref{eq:mismatchbias}. For $p_n \leq n^c$ for some $c>0$, this factor behaves as a power of $n$. 
\end{enumerate}

These observations suggest that $\tau = \frac{p_n}{n}\sqrt{\log (n/p_n)}$ is a good choice, because then \eqref{eq:mseupperbound}, \eqref{eq:varupperbound}, \eqref{eq:varlowerbound}, \eqref{eq:mismatchbias}, \eqref{eq:mismatchvar} are all of the order $p_n\log(n/p_n)$, suggesting that the posterior contracts at the minimax rate around both the truth and the horseshoe estimator.

\section{\label{sec:empiricalbayes}Empirical Bayes estimation of $\tau$}
A natural follow-up question is how to choose $\tau$ in practice, when $p_n$ is unknown. As  discussed in section \ref{sec:horseshoe}, the full Bayesian approach suggested by \cite{Carvalho2010} performs well in simulations. The analysis of such a  hierarchical prior would however require different tools than the ones we have used so far. An empirical Bayes estimate of $\tau$ would be a natural solution, and allows us in practice to use one of the representations in (\ref{eq:pmexpression}) for computations, instead of an MCMC-type algorithm.

By adapting the approach in paragraph 6.2 in \citep{Johnstone2004}, we can find conditions under which the horseshoe estimator with an empirical Bayes estimate of $\tau$ will still  attain the minimax $\ell_2$ risk. Based on the consideration of Section \ref{sec:results}, we proceed with the choices $\tau = \frac{p_n}{n}\sqrt{\log(n/p_n)}$ and $\tau = \frac{p_n}{n}$. The former is optimal in the sense that the posterior spread is of the order of the minimax risk, but the latter has the simple interpretation of being the proportion of nonzero means, and the difference between the two is only the square root of a log factor.\\

\begin{thm}\label{thm:empiricalbayes}
Suppose we observe an $n$-dimensional vector $Y \sim \mathcal{N}(\theta_0, \sigma^2I_n)$ and we use $T_{\widehat \tau}(y)$ as our estimator of $\theta_0$. If $\widehat \tau \in (0, 1)$ satisfies the following two conditions for $\tau = \frac{p_n}{n}$ or $\tau = \frac{p_n}{n}\sqrt{\log(n/p_n)}$:
\begin{enumerate}
\item $\mathbb{P}_{\theta_0}\left(\widehat\tau > c\tau\right)  \lesssim \frac{p_n}{n}$ for a constant $c \geq 1$ such that $\tau \leq \frac{1}{c}$;
\item There exists a function $g: \mathbb{N} \times \mathbb{N} \to (0, 1)$ such that $\widehat\tau \geq g(n, p_n)$ with probability one and $-\log (g(n, p_n)) \mathbb{P}_{\theta_0}\left(\widehat\tau \leq\tau \right) \lesssim \log(n/p_n) $,
\end{enumerate}
then:
\begin{equation}\label{eq:empbayesresult}
\sup_{\theta_0 \in \ell_0[p_n]} \E_{\theta_0} \|T_{\widehat \tau}(Y) - \theta_0\|^2 \asymp p_n\log \frac{n}{p_n}
\end{equation}
 as $n, p_n \to \infty$ and $p_n = o(n)$.  If only the first condition can be verified for an estimator $\widehat\tau$, then $\sup\{\frac{1}{n}, \widehat\tau\}$ will have an $\ell_2$ risk of at most order $p_n \log n$.
\end{thm}

The first condition requires that $\widehat\tau$ does not overestimate the fraction $\frac{p_n}{n}$ of nonzero means (up to a log factor) too much or by a too large probability.  If $p_n \geq 1$, as we have assumed, then it is satisfied already by $\widehat\tau = \frac{1}{n}$ (and $c=1$). According to the last assertion of the theorem, this `universal threshold' yields the rate $p_n\log n$ (possibly up to a multiplicative constant). This is equal to the rate of the Lasso estimator with the usual choice of $\lambda = 2\sqrt{2\sigma^2 \log n}$ \citep{Bickel2009}. However, in the framework where $p_n \to \infty$, the estimator $\widehat\tau = \frac{1}{n}$ will certainly underestimate the sparsity level. A more natural estimator of $\frac{p_n}{n}$ is:
\begin{equation}\label{eq:simpleestimator}
\widehat\tau =  \frac{\#\{|y_i| \geq \sqrt{c_1\sigma^2\log n}, i = 1, \ldots, n \}}{c_2n},
\end{equation}
where $c_1$ and $c_2$ are positive constants. By Lemma \ref{lem:schatter}, this estimator satisfies the first condition for $\tau = \frac{p_n}{n}$ and $\tau = \frac{p_n}{n}\sqrt{\log(n/p_n)}$ if $c_1 > 2, c_2 > 1$ and $p_n \to \infty$ or $c_1 = 2, c_2 > 1$ and $p_n \gtrsim \log n$. Thus $\max\left\{\widehat\tau, \frac{1}{n}\right\}$ will also lead to a rate of at most order $p_n\log n$ under these conditions. Its behaviour will be explored further in section \ref{sec:simulation}.

The rate can be improved to $p_n\log (n/p_n)$ if the second condition is met as well, which ensures that the sparsity level is not underestimated too much or by a too large probability. As we are not aware of any estimators meeting this condition for all $\theta_0$, this condition is currently mostly of theoretical interest. If the true mean vector is very sparse, in the sense that there are relatively few nonzero means or the nonzero means are close to zero, there is not much to be gained in terms of rates by meeting this condition.  The extra occurrence of $p_n$  relative to the rate $p_n\log n$ is of interest only if $p_n$ is relatively large. For instance, if $p_n \asymp n^\alpha$ for $\alpha \in (0, 1)$, then $p_n \log(n/p_n) = (1-\alpha)p_n\log n$, which suggests a decrease of the proportionality constant in (\ref{eq:empbayesresult}), particularly if $\alpha$ is close to one. Furthermore, when $p_n$ is large, the constant in \eqref{eq:empbayesresult} may be sensitive to the fine properties of $\widehat\tau$, as it depends on $g(n, p_n)$ (as can be seen in the proof). If $\widehat\tau$ seriously underestimates the sparsity level, the corresponding value of $g(n, p_n)$ from the second condition may be so small that  the upper bound on the multiplicative constant before (\ref{eq:empbayesresult}) becomes very large. Hence in this case, $\widehat\tau$ is required to be close to the proportion $\frac{p_n}{n}$ (up to a log factor) with large probability in order to get an optimal rate.

\cite{Datta2013} warned against the use of an empirical Bayes estimate of $\tau$ for the horseshoe prior, because the estimate might collapse to zero. Their references for this statement, \cite{Scott2010} and \cite{Bogdan2008}, indicate that they are thinking of a marginal maximum likelihood estimate of $\tau$. However, an empirical Bayes estimate of $\tau$ does not need to be based on this principle. Furthermore, an estimator that satisfies the second condition from Theorem \ref{thm:empiricalbayes} or that is truncated from below by $\frac{1}{n}$, would not be susceptible to this potential problem. 

\section{\label{sec:simulation}Simulation study}

A simulation study provides more insight into the behaviour of the horseshoe estimator, both when using an empirical Bayes procedure with estimator (\ref{eq:simpleestimator}) and when using the fully Bayesian procedure proposed by \cite{Carvalho2010} with a half-Cauchy prior on $\tau$. For each data point, 100 replicates of an $n$-dimensional vector sampled from a $\mathcal{N}(\theta_0, I _n)$ distribution were created, where $\theta_0$ had either 20, 40 or 200 (5\%, 10\% or 50\%) entries equal to an integer $A$ ranging from 1 to 10, and all the other entries equal to zero. The full Bayesian version was implemented using the code provided in \citep{Scott2010-2}, and the coordinatewise posterior mean was used as the estimator of $\theta_0$. For the empirical Bayes procedure, the estimator (\ref{eq:simpleestimator}) was used with $c_1 = 2$ and $c_2 = 1$. Performance was measured by squared error loss, which was averaged across replicates to create Figure \ref{fig:simulationresults}. 

\begin{figure}[ht!]
\begin{center}
\subfigure[]
{
\includegraphics[scale = 0.26]{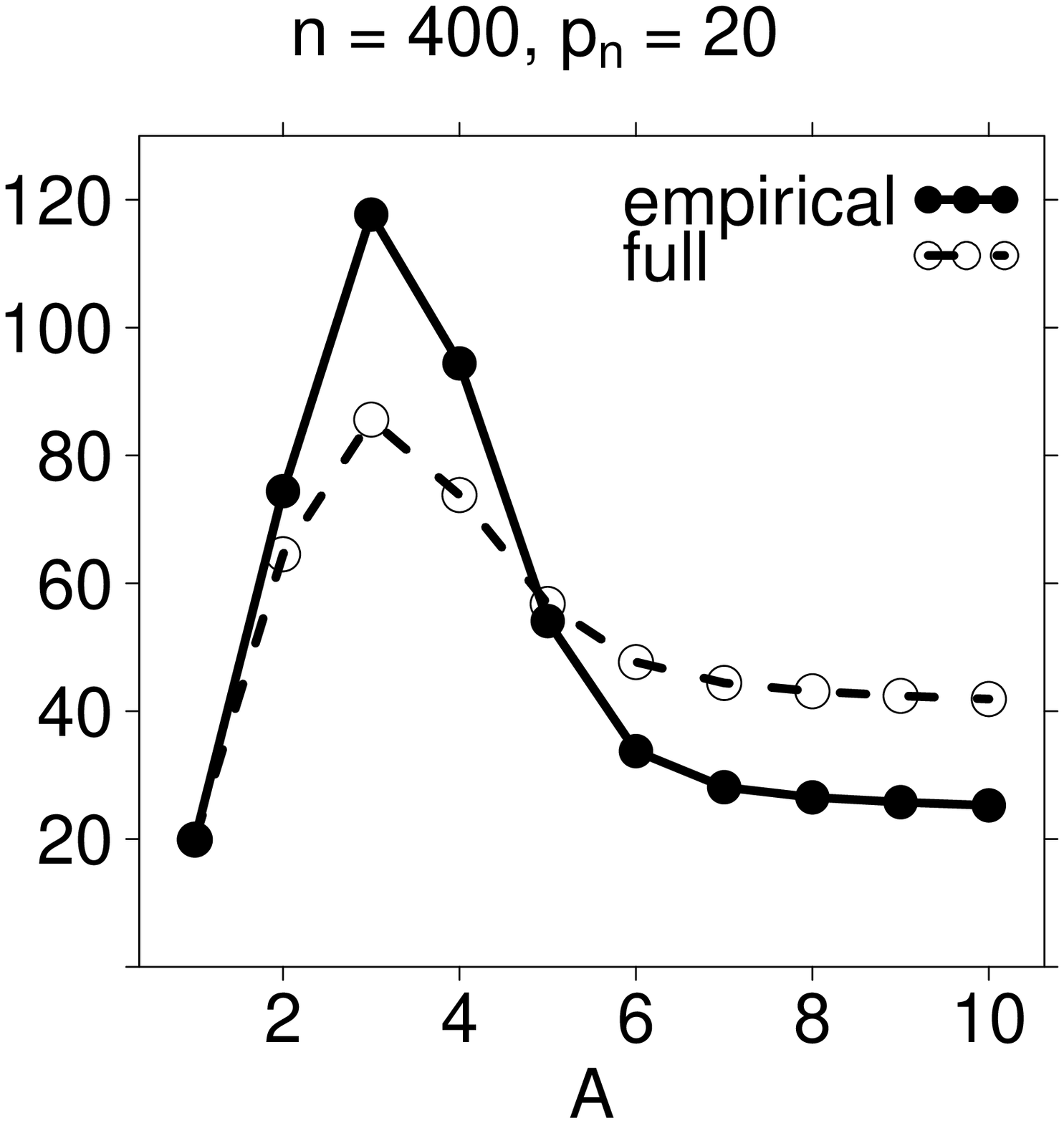} 
\label{fig:sim1}
}
\subfigure[]
{
\label{fig:sim2}
\includegraphics[scale=0.26]{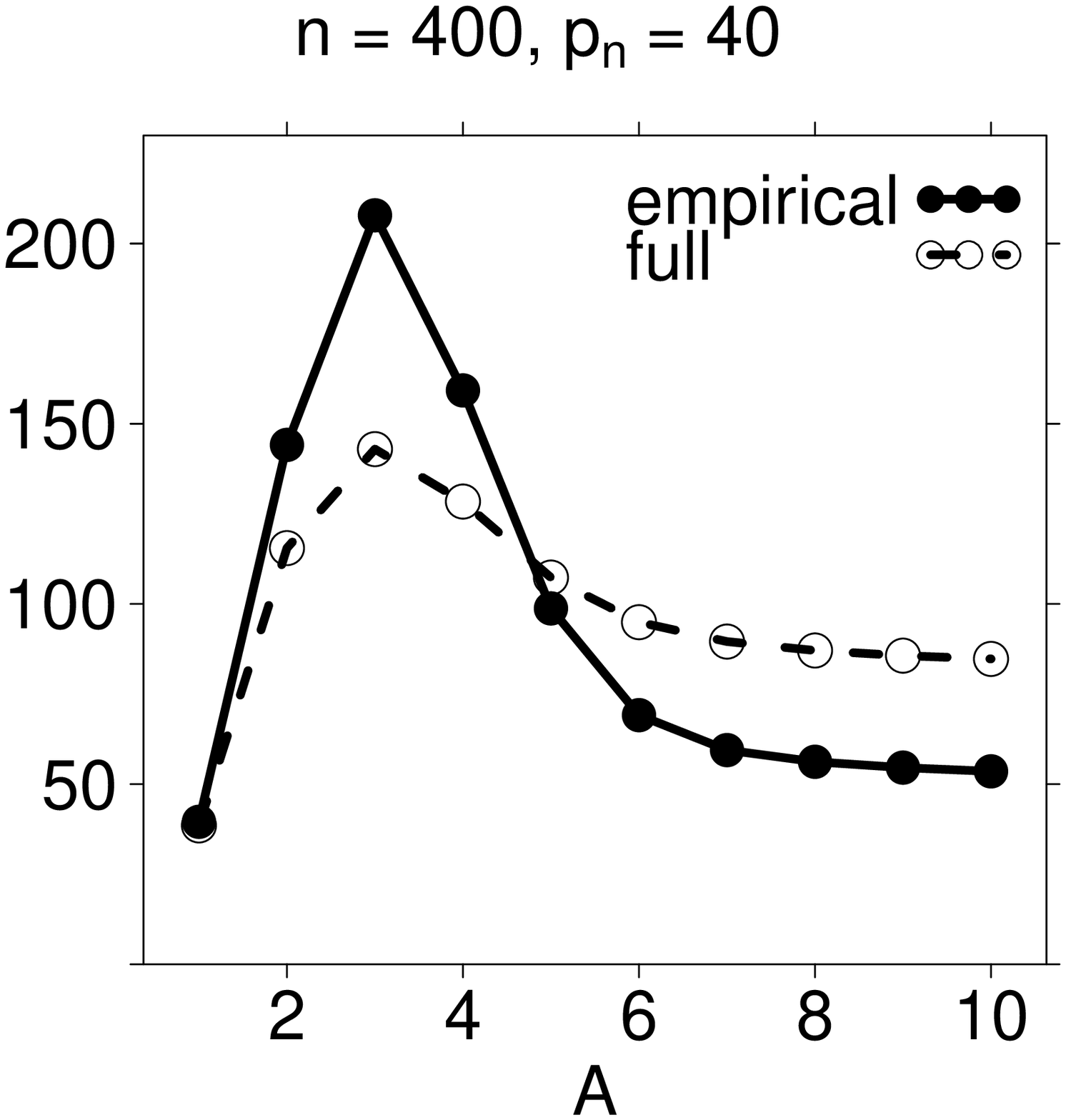} 
}
\subfigure[]
{
\label{fig:sim3}
\includegraphics[scale=0.26]{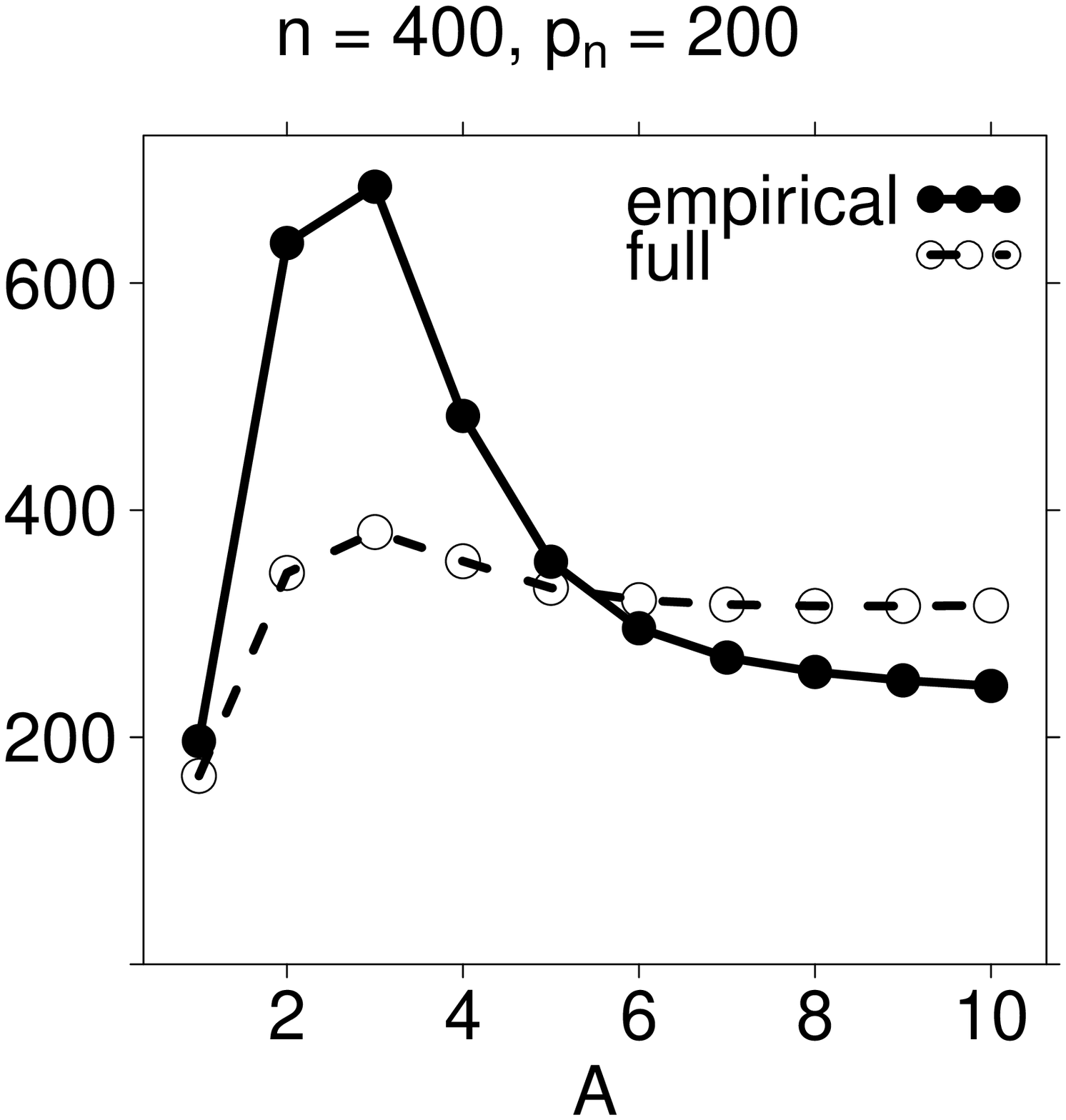} 
}
\subfigure[]
{
\label{fig:sim4}
\includegraphics[scale=0.26]{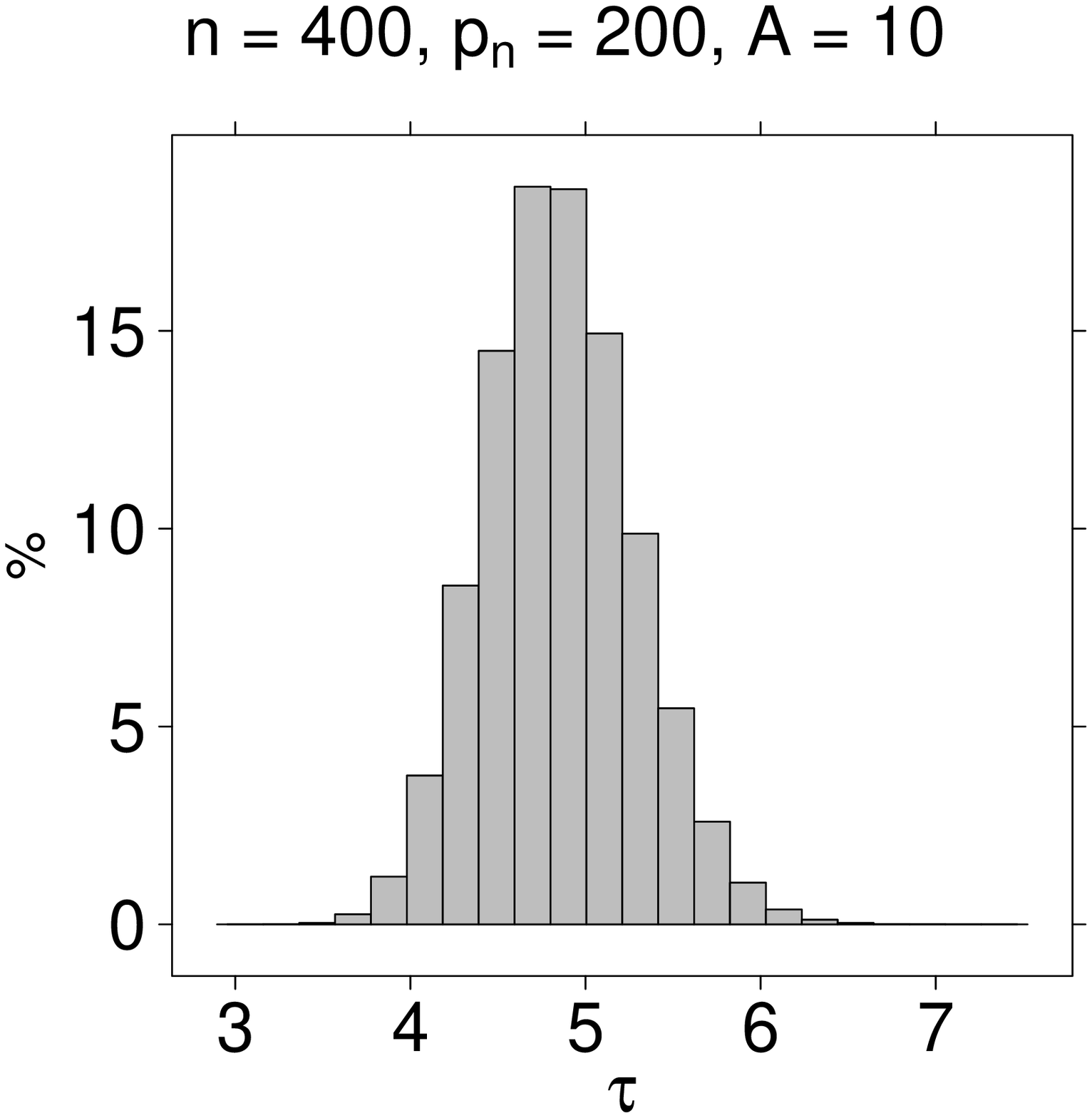} 
}
\caption{Average squared error loss over 100 replicates with underlying mean vectors of length $n = 400$ if the nonzero coefficients are taken equal to $A$, in case 5\% (Figure (a)), 10\% (Figure (b)) or 50\% (Figure (c)) of the means are equal to a nonzero value $A$. The solid line corresponds to empirical Bayes with (\ref{eq:simpleestimator}), $c_1 = 2, c_2=1$, the dashed line to full Bayes with a half-Cauchy prior on $\tau$. Figure (d) displays a histogram of all Gibbs samples of $\tau$ (after the burn-in) of all replicates in the setting $\tau \sim C^+(0,1)$, $A = 10$, $p_n = 200$.   }
\label{fig:simulationresults}
\end{center}
\end{figure}

In all settings, both estimators experience a peak in the $\ell_2$ loss for values of $A$ close to the `universal threshold' of $\sqrt{2\log 400} \approx 3.5$. This is not unexpected, as in the terminology of \cite{Johnstone2004}, the horseshoe estimator is a shrinkage rule, and while it is not a thresholding rule in their sense, it does have the bounded shrinkage property which leads to thresholding-like behaviour. The bounded shrinkage property can be derived from Lemma \ref{lem:diffbound}, which yields the following inequality as $\tau$ approaches zero:
\begin{equation*}
|T_\tau(y) - y| \leq \sqrt{2\sigma^2\log\frac{1}{\tau}}.
\end{equation*}
With $\tau = \frac{1}{n}$, this leads to the `universal threshold' of $\sqrt{2\sigma^2\log n}$, or with $\tau = \left(\frac{p_n}{n}\right)^\alpha$, a `threshold' at $\sqrt{2\alpha\sigma^2\log(n/p_n)}$. Based on this property and the proofs of the main results, we can divide the underlying parameters into three cases: 
\begin{enumerate}[(i)]
\item Those that are exactly or close to zero, where the observations are shrunk close to zero; 
\item Those that are larger than the threshold, where the horseshoe estimator essentially behaves like the identity;
\item Those that are close to the `threshold', where the horseshoe estimator is most likely to shrink the observations too much. 
\end{enumerate}
The horseshoe estimator performs well in cases (i) and (ii) due to its pole at zero and its heavy tails respectively. The hardest parameters to recover from the noise are those that are close to the threshold, and these are the ones that affect the estimation risk the most. This phenomenon explains the peaks in the graphs of Figure \ref{fig:simulationresults} around $A = 3.5$. 

The full Bayes implementation with a Cauchy prior on $\tau$ attains a lower $\ell_2$ loss around the universal threshold than the empirical Bayes procedure. This is because estimator (\ref{eq:simpleestimator}) counts the number of observations that are above the universal threshold. When all the nonzero means are close to this threshold, $\widehat\tau$ may `miss' some of them, thereby underestimating the sparsity level $\frac{p_n}{n}$ and thus leading to overshrinkage.

For values of $A$ well past the universal threshold, the empirical Bayes estimator does better than the full Bayes version. For such large values of $A$, the estimator (\ref{eq:simpleestimator}) will be equal to the true sparsity level with large probability and hence its good performance is not unexpected. However, an interesting question is why the full Bayes estimator does not do as well as the empirical Bayes estimator, especially because the nonzero means are so far removed from zero that the problem is `easy'. This could be due to the choice of a half-Cauchy prior for $\tau$: it places no restriction on the possible values of $\tau$ and has such heavy tails that values far exceeding the sparsity level $\frac{p_n}{n}$ are possible. This would lead to undershrinkage of the observations corresponding to a zero mean, which would be reflected in the $\ell_2$ loss. Figure \ref{fig:sim4} shows a histogram of all Gibbs samples of $\tau$ in the setting where 50\% of the means are set equal to 10. The range of these values is (3.1, 7.3), which is very far away from $\frac{p_n}{n} = \frac{1}{2}$. This indicates that a full Bayesian version of the horseshoe prior could benefit from a different choice of prior on $\tau$ than a half-Cauchy one, for example one that is restricted to [0,1].

\section{\label{sec:finalremarks}Concluding remarks} 
The choice of the global shrinkage parameter $\tau$ is critical towards ensuring the right amount of shrinkage of the observations to recover the underlying mean vector. The value of  $\tau = \frac{p_n}{n}\sqrt{\log(n/p_n)}$ was found to be optimal.  Theorem \ref{thm:empiricalbayes} indicates that quite a wide range of estimators for $\tau$ will work well, especially in cases where the underlying mean vector is sparse. Of course, it should not come as a surprise that an estimator designed to recover sparse vectors will work especially well if the truth is indeed sparse. An interesting extension to this work would be to investigate whether the posterior concentration properties of the horseshoe prior still remain when a hyperprior is placed on $\tau$. The result that $\tau = \frac{p_n}{n}$ (up to a log factor) yields optimal rates, together with the simulation results, suggests that in a fully Bayesian approach, a prior on $\tau$ which is restricted to $[0, 1]$ may perform  better than the suggested half-Cauchy prior. 

The simulation results also indicate that mean vectors with the nonzero means close to the universal threshold are the hardest to recover. In future simulations involving shrinkage rules, it would therefore be interesting to study the challenging case where all the nonzero parameters are at this threshold. The performance of the empirical Bayes estimator (\ref{eq:simpleestimator}) leaves something to be desired around the threshold. In additional numerical experiments (not shown), we tried two other estimators of $\tau$. The first was the `oracle estimator' $\widehat\tau = \frac{p_n}{n}$. For values of the nonzero means well past the `threshold', the behaviour of this estimator was very similar to that of (\ref{eq:simpleestimator}). However, before the threshold, the squared error loss of the empirical procedure with the oracle estimator was between that of the full Bayes estimator and empirical Bayes with estimator (\ref{eq:simpleestimator}). The second estimator was the mean of the samples of $\tau$ from the full Bayes estimator. The resulting squared error loss was remarkably close to that of the full Bayes estimator, for all values of the nonzero means. Neither of these two estimators is of much practical use. However, their range of behaviours suggests room for improvement over the estimator (\ref{eq:simpleestimator}), and it would be worthwhile to study more refined estimators for $\tau$. 

An interesting question is what aspects of the horseshoe prior are truly essential towards optimal posterior contraction properties.  Our proofs do not elucidate whether the pole at zero of the horseshoe prior is required, or if a prior with heavy tails, and in a sense `sufficient' mass at zero would work as well. The failure of the Lasso to concentrate around the true mean vector at the minimax rate does indicate that heavy tails in itself may not be sufficient, and  adding mass at zero solves this problem \citep{Castillo2013, Castillo2012}. It is possible that the pole at zero is inessential, in particular if the global tuning parameter is chosen carefully, for instance by empirical Bayes. If the tuning parameter is chosen by a full Bayes method, the peak may be more essential, depending on its prior.

The horseshoe estimator has the  property that its computational complexity depends on the sparsity level rather than  the number of observations. Although there is no point mass at zero to induce sparsity, it still yields good reconstruction in $\ell_2$, and a posterior distribution that contracts at an informative rate. None of the estimates will however be exactly zero. Variable selection can be performed by applying some sort of thresholding rule, such as the one suggested in \citep{Carvalho2010} and analyzed by \cite{Datta2013}. The performance of this thresholding rule in simulations in the two works cited has been encouraging.

\section*{Acknowledgements}
The authors would like to thank two anonymous referees for their helpful suggestions, as well as James Scott for his advice on implementing the full Bayesian version of the horseshoe estimator.

\appendix 
\section{\label{sec:proofs}Proofs} 
This section begins with Lemma \ref{lem:integralbounds}, providing bounds on some of the degenerate hypergeometric functions appearing in the posterior mean and posterior variance. This is followed by 
two lemmas that are needed for the proofs of Theorems \ref{thm:mse} and \ref{thm:var}: Lemma \ref{lem:postmeanupperbounds} provides two upper bounds on the horseshoe estimator and Lemma \ref{lem:diffbound} gives a bound on the absolute value of the difference between the horseshoe estimator and an observation. We then proceed to the proof of Theorem \ref{thm:mse}, after which Lemma \ref{lem:postvar} provides upper bounds on the posterior variance. These upper bounds are then used in the proof of Theorem \ref{thm:var}. The proof of Theorem \ref{thm:varlowerbound} is given next, followed by Lemmas \ref{lem:mismatch1} and \ref{lem:mismatch2} supporting  the proof of Theorem \ref{thm:mismatch}. This section concludes with the proofs of Theorem \ref{thm:empiricalbayes} and Lemma \ref{lem:schatter}, which both concern the empirical Bayes procedure discussed in section \ref{sec:empiricalbayes}.\\

\begin{lem}\label{lem:integralbounds}
Define\begin{equation*}
I_k(y) := \int_0^1 z^k \frac{1}{\tau^2 + (1- \tau^2)z}e^{\frac{y^2}{2\sigma^2} z} dz.
\end{equation*}
Then, for $a > 1$: 
\begin{align}
I_{\frac{3}{2}}(y) &\geq \frac{1}{5}\tau^3 + \sigma^2\frac{\tau}{y^2} \left(e^{\frac{y^2}{2a\sigma^2}} - e^{\tau^2\frac{y^2}{2\sigma^2}} \right) + \frac{\sigma^2}{\sqrt{a}y^2}\left(e^\frac{y^2}{2\sigma^2} - e^{\frac{y^2}{2a\sigma^2}}\right),\label{eq:lower3/2}\\
I_{\frac{1}{2}}(y) &\geq \frac{1}{3}\tau + \frac{\sigma^2}{y^2}\left(e^\frac{y^2}{2\sigma^2} - e^{\tau^2\frac{y^2}{2\sigma^2}}\right), \label{eq:lower1/2}\\
I_{\frac{1}{2}}(y) &\leq \frac{2}{3}e^{\tau^2\frac{y^2}{2\sigma^2}}\tau + 2e^\frac{y^2}{2a\sigma^2}\left(\frac{1}{\sqrt{a}} - \tau\right) + \frac{2\sqrt{a}\sigma^2}{y^2}\left(e^\frac{y^2}{2\sigma^2} - e^\frac{y^2}{2a\sigma^2}\right),\label{eq:upper1/2}\\
I_{-\frac{1}{2}}(y) &\geq \frac{1}{\tau} + e^{\tau^2\frac{y^2}{2\sigma^2}}\left(\frac{1}{\tau} - \frac{1}{\sqrt{\tau}}\right) + \frac{a\sqrt{a}\sigma^2}{y^2}\left(e^\frac{y^2}{2a\sigma^2} - e^{\tau\frac{y^2}{2\sigma^2}}\right)\notag\\
&\quad + \frac{\sigma^2}{y^2}\left(e^\frac{y^2}{2\sigma^2} - e^\frac{y^2}{2a\sigma^2}\right), \label{eq:lower-1/2}\\
I_{-\frac{1}{2}}(y) & \leq \frac{2e^{\tau^2\frac{y^2}{2\sigma^2}}}{\tau} + 2e^{\tau\frac{y^2}{2\sigma^2}}\left(\frac{1}{\tau} - \frac{1}{\sqrt{\tau}}\right) + 2e^\frac{y^2}{2a\sigma^2}\left(\frac{1}{\sqrt{\tau}} - \sqrt{a}\right) \notag \\
&\quad + \frac{2a\sqrt{a}\sigma^2}{y^2}\left(e^\frac{y^2}{2\sigma^2} - e^\frac{y^2}{2a\sigma^2}\right), \label{eq:upper-1/2}
\end{align}
where \eqref{eq:lower3/2} and \eqref{eq:upper1/2} hold for $\tau < 1/\sqrt{a}$, \eqref{eq:lower1/2} holds for $\tau < 1$, and \eqref{eq:lower-1/2} and \eqref{eq:upper-1/2} hold for $\tau < 1/a$.
\end{lem}

\begin{proof}
Write $\xi = y^2/(2\sigma^2)$. We first note that for $z \geq \tau^2$, we have $z \leq \tau^2 + (1-\tau^2)z \leq 2z$, while for $z \leq \tau^2$, we have $\tau^2 \leq \tau^2 + (1-\tau^2)z \leq 2\tau^2$. Hence, we can bound $I_k$ from above by:
\begin{equation*}
I_k(y) \leq \frac{1}{\tau^2} \int_0^{\tau^2} z^k e^{\xi z} dz + \int_{\tau^2}^1 z^{k-1} e^{\xi z} dz,
\end{equation*}
and from below by half of that quantity. We bound the integral over $[0, \tau^2]$ in all cases by bounding the factor $e^{\xi z}$ by 1 or $e^{\tau^2\xi}$. For the integral over $[\tau^2, 1]$, we first substitute $u = \xi z$, yielding: $\int_{\tau^2}^1 z^{k-1} e^{\xi z} dz = \xi^{-k}\int_{\tau^2 \xi}^\xi u^{k-1} e^u du$. For \eqref{eq:lower3/2} and \eqref{eq:upper1/2}, we split the domain of integration into $\left[\tau^2\xi, \frac{\xi}{a}\right]$ and $\left[\frac{\xi}{a}, \xi\right]$. For $I_\frac{3}{2}$, we bound by:
\begin{equation*}
I_{\frac{3}{2}}(y) \geq \frac{1}{2}\left(
\frac{1}{\tau^2} \int_0^{\tau^2} z^\frac{3}{2}dz
+
\xi^{-\frac{3}{2}} (\tau^2 \xi)^\frac{1}{2} \int_{\tau^2\xi}^\frac{\xi}{a} e^u du
+
\xi^{-\frac{3}{2}} \left(\frac{\xi}{a}\right)^\frac{1}{2} \int_\frac{\xi}{a}^\xi e^u du \right),
\end{equation*}
yielding (\ref{eq:lower3/2}). Similarly, for $I_\frac{1}{2}$:
\begin{equation*}
I_\frac{1}{2}(y) \leq
\frac{1}{\tau^2} e^{\tau^2\xi} \int_0^{\tau^2} z^\frac{1}{2}dz
+
\xi^{-\frac{1}{2}} e^\frac{\xi}{a} \int_{\tau^2 \xi}^\frac{\xi}{a} u^{-\frac{1}{2}}du
+
\xi^{-\frac{1}{2}} \left(\frac{\xi}{a}\right)^{-\frac{1}{2}} \int_\frac{\xi}{a}^\xi e^u du,
\end{equation*}
resulting in (\ref{eq:upper1/2}). The bound \eqref{eq:lower1/2} is obtained similarly, but without splitting up $[\tau^2\xi, \xi]$ further, by the inequality
\begin{equation*}
I_\frac{1}{2}(y) \geq \frac{1}{2\tau^2}\int_0^{\tau^2} z^\frac{1}{2}dz + \frac{1}{2}\xi^{-1}\int_{\tau^2\xi}^\xi e^u du.
\end{equation*}
For the bounds on $I_{-\frac{1}{2}}$, we split up the domain of integration $\left[\tau^2\xi, \xi\right]$ into $\left[\tau^2\xi, \tau\xi\right], \left[\tau\xi, \frac{\xi}{a}\right]$ and $\left[\frac{\xi}{a}, \xi\right]$, and then bound by:
\begin{align*}
I_{-\frac{1}{2}}(y) &\geq \frac{1}{2}\left(
\frac{1}{\tau^2} \int_0^{\tau^2} z^{-\frac{1}{2}}dz
+
\xi^\frac{1}{2} e^{\tau^2\xi} \int_{\tau^2\xi}^{\tau\xi} u^{-\frac{3}{2}}du
+
\xi^\frac{1}{2} \left(\frac{\xi}{a}\right)^{-\frac{3}{2}}\int_{\tau\xi}^\frac{\xi}{a} e^u du \right.\\
&\quad+ \left. 
\xi^\frac{1}{2} \xi^{-\frac{3}{2}} \int_\frac{\xi}{a}^\xi e^u du
\right),
\end{align*}
yielding (\ref{eq:lower-1/2}), and by:
\begin{align*}
I_{-\frac{1}{2}}(y) &\leq
\frac{1}{\tau^2}e^{\tau^2\xi}\int_0^{\tau^2} z^{-\frac{1}{2}}dz
+
\xi^\frac{1}{2} e^{\tau\xi} \int_{\tau^2\xi}^{\tau\xi} u^{-\frac{3}{2}}du
+ 
\xi^\frac{1}{2} e^\frac{\xi}{a} \int_{\tau\xi}^\frac{\xi}{a} u^{-\frac{3}{2}}du\\
&\quad+
\xi^\frac{1}{2} \left(\frac{\xi}{a}\right)^{-\frac{3}{2}} \int_{\frac{\xi}{a}}^\xi e^u du,
\end{align*}
to find (\ref{eq:upper-1/2}).
\end{proof}

\begin{lem}\label{lem:postmeanupperbounds}
If $\tau^2 < 1$, the posterior mean of the horseshoe prior can be bounded above by:
\begin{enumerate}
\item $T_\tau(y) \leq ye^{\frac{y^2}{2\sigma^2}}f(\tau)$, where $f$ is such that $ f(\tau) \leq \frac{2}{3}\tau$;

\item $T_\tau(y) \leq y \frac{\frac{2}{3}e^{\tau^2\frac{y^2}{2\sigma^2}}\tau + 2e^{\frac{y^2}{2a\sigma^2}}\left(\frac{1}{\sqrt{a}} - \tau\right) + \frac{2\sqrt{a}\sigma^2}{y^2}\left(e^\frac{y^2}{2\sigma^2}- e^\frac{y^2}{2a\sigma^2}\right)  }
{\frac{1}{\tau} + e^{\tau^2\frac{y^2}{2\sigma^2}}\left(\frac{1}{\tau} - \frac{1}{\sqrt{\tau}}\right) + \frac{a\sigma^2\sqrt{a}}{y^2}\left(e^\frac{y^2}{2a\sigma^2} - e^{\tau\frac{y^2}{2\sigma^2}}\right) + \frac{\sigma^2}{y^2}\left(e^\frac{y^2}{2\sigma^2} - e^\frac{y^2}{2a\sigma^2}\right) }$
, for any $a > 1$ and $\tau < \tfrac{1}{a}$. 

\end{enumerate}

\end{lem}

\begin{proof}
We bound the integrals in the numerator and denominator of expression (\ref{eq:pmexpression}). For the first upper bound, we will use the fact that for $0 \leq z \leq 1$, $e^{\frac{y^2}{2\sigma^2}z}$ is bounded below by 1 and above by $e^{\frac{y^2}{2\sigma^2}}$. The posterior mean can therefore be bounded by:
\begin{equation*}
T_\tau (y) \leq ye^{\frac{y^2}{2\sigma^2}}
\frac{ \int_0^1 z^{\frac{1}{2}}\frac{1}{\tau^2 + \left(1 - \tau^2 \right)z } dz  } {\int_0^1 z^{-\frac{1}{2}}\frac{1}{\tau^2 + \left(1 - \tau^2\right)z } dz   }
= ye^{\frac{y^2}{2\sigma^2}}f(\tau),
\end{equation*}
where
\begin{equation*}
f(\tau) = \frac{\tau}{1-\tau^2}\left(\frac{\sqrt{1-\tau^2}}{\arctan\left(\frac{\sqrt{1-\tau^2}}{\tau} \right)} - \tau \right).
\end{equation*}
By Shafer's inequality for the arctangent \citep{Shafer1966}: 
\begin{align*}
\frac{f(\tau)}{\tau} = \frac{1}{1-\tau^2}\left(\frac{\sqrt{1-\tau^2}}{\arctan\left(\frac{\sqrt{1-\tau^2}}{\tau} \right)} - \tau \right) < \frac{2}{3}\frac{1}{1 + \tau} \leq \frac{2}{3},
\end{align*}
which completes the proof for the first upper bound.\\
\\
 For the second inequality, we note that, in the notation of Lemma \ref{lem:integralbounds}, $ T_\tau(y) = y\frac{I_\frac{1}{2}(y)}{I_{-\frac{1}{2}}(y)}$. The bounds in Lemma \ref{lem:integralbounds} yield the stated inequality. 
\end{proof}

\begin{lem}\label{lem:diffbound}
For $\tau^2 < 1$, the absolute value of the difference between the horseshoe estimator and an observation $y$ can be bounded by a function $h(y, \tau)$ such that for any $c > 2$:
\begin{equation*}
\lim_{\tau \downarrow 0} \sup_{|y| > \sqrt{c\sigma^2 \log \frac{1}{\tau}}} h(y, \tau) = 0.
\end{equation*}
\end{lem}

\begin{proof}
We assume $y > 0$ without loss of generality. By a change of variables of $x = 1 - z$:
\begin{equation*}
|T_\tau(y) - y| 
= y \frac{ \int_0^1 e^{-\frac{y^2}{2\sigma^2}x}x(1-x)^{-\frac{1}{2}} \frac{1}{1 - (1-\tau^2)x} dx   } {\int_0^1 e^{-\frac{y^2}{2\sigma^2}x}(1-x)^{-\frac{1}{2}} \frac{1}{1 - (1-\tau^2)x} dx   }.
\end{equation*}
By following the proof of Watson's lemma provided in \cite{Miller2006}, we can find bounds on the numerator and denominator of the above expression. First define $g(x) = (1-x)^{-\frac{1}{2}} \frac{1}{\1 - (1-\tau^2)x}$ and note that by Taylor's theorem, $g(x) = g(0) + xg'(\xi_x),
$ where $\xi_x$ is between 0 and $x$. Let $s$ be any number between 0 and 1. Because $g''(x)$ is not negative for $x \in [0, 1)$, we have that for $x \in [0, s]$, $s \in (0, 1)$: $g'(0) \leq g'(x) \leq g'(s)$. The numerator can then be bounded by:
\begin{align*}
\int_0^1 e^{-\frac{y^2}{2\sigma^2}x}xg(x) dx 
&= \int_0^s e^{-\frac{y^2}{2\sigma^2}x}xg(0)dx + \int_0^s e^{-\frac{y^2}{2\sigma^2}x}x^2 g'(\xi_x)dx\\
&\quad + \int_s^1 e^{-\frac{y^2}{2\sigma^2}x}xg(x) dx\\
&\leq \frac{1}{y^4}h_1(y, \sigma, s) + \frac{g'(s)}{y^6}h_2(y, \sigma, s) + 2e^{-\frac{sy^2}{2\sigma^2}} h_3(\tau),
\end{align*}
where $h_1(y, \sigma, s) = 4\sigma^4 - 2\sigma^2(sy^2 + 2\sigma^2)e^{-\frac{sy^2}{2\sigma^2}}$, $h_2(y, \sigma, s) = 16\sigma^6 - 2\sigma^2(s^2y^4 + 4s\sigma^2y^2 + 8\sigma^4)e^{-\frac{sy^2}{2\sigma^2}}$ and $h_3(\tau) =  \arctan\left(\frac{\sqrt{1-\tau^2}}{\tau}\right)\tau^{-1}(1-\tau^2)^{-\frac{3}{2}} - (1-\tau^2)^{-1}$. The denominator can similarly be bounded by:
\begin{align*}
\int_0^1 e^{-\frac{y^2}{2\sigma^2}x}g(x)dx 
&= \int_0^s e^{-\frac{y^2}{2\sigma^2}x}g(0)dx + \int_0^s e^{-\frac{y^2}{2\sigma^2}x}x g'(\xi_x)dx\\ 
&\quad + \int_s^1 e^{-\frac{y^2}{2\sigma^2}x}g(x) dx\\
&\geq \frac{1}{y^2}h_4(y, \sigma, s) + \frac{g'(0)}{y^4}h_5(y,\sigma,s) + 0,
\end{align*}
where $h_4(y, \sigma, s) = 2\sigma^2 - 2\sigma^2e^{-\frac{sy^2}{2\sigma^2}}$ and $h_5(y, \sigma, s) = 4\sigma^4 - 2\sigma^2e^{-\frac{sy^2}{2\sigma^2}}(sy^2 + 2\sigma^2)$. Hence:
\begin{equation*}
|T_\tau(y) - y| \leq \frac{\frac{1}{y}h_1(y, \sigma, s) + \frac{g'(s)}{y^3}h_2(y, \sigma, s) + 2y^3e^{-\frac{sy^2}{2\sigma^2}} h_3(\tau) }{h_4(y, \sigma, s) + \frac{g'(0)}{y^2}h_5(y, \sigma, s)  }.
\end{equation*}
For any fixed $\tau$, this bound tends to zero as $y$ tends to infinity. If $\tau \to 0$, the term containing $h_3(\tau)$ could potentially diverge. For $s = \frac{2}{3}$ and $y = \sqrt{c\sigma^2\log(1/\tau)}$, where $c$ is a positive constant, this term displays the following limiting behaviour as $\tau \to 0$:
\begin{align*}
\lim_{\tau \downarrow 0}y^3 e^{-\frac{1}{3\sigma^2}y^2}h_{3}(\tau)
&= \lim_{\tau \downarrow 0} \left(c\sigma^2\log\frac{1}{\tau} \right)^{\frac{3}{2}}\tau^{\frac{c}{3}-1}\left(\frac{\arctan\left(\frac{\sqrt{1-\tau^2}}{\tau}\right) }{(1-\tau^2)^\frac{3}{2}} - \frac{\tau}{1-\tau^2} \right)\\
&= \begin{cases}
0 & c > 3\\
\infty & \text{otherwise},
\end{cases}
\end{align*}
because $\lim_{\tau \downarrow 0} \arctan\left(\frac{\sqrt{1-\tau^2}}{\tau}\right) (1-\tau^2)^{-\frac{3}{2}} = \frac{\pi}{2}$, $\lim_{\tau \downarrow 0} \frac{\tau}{1-\tau^2} = 0$ and  the factor $\left(c\sigma^2\log(1/\tau) \right)^{\frac{3}{2}}  \tau^{ \frac{c}{3} - 1}$ tends to zero as $\tau \downarrow 0$ if $\frac{c}{3}-1> 0$ and infinity otherwise. The condition $c > 3$ is related to the choice of $s = \frac{2}{3}$ and can be improved to any constant strictly greater than $2$ by choosing $s$ appropriately close to one.  Hence, we find that the absolute value of the difference between the posterior mean and an observation can be bounded by a function $h(y, \tau)$ with the desired property.
\end{proof}

\textbf{Proof of Theorem \ref{thm:mse} }

\begin{proof}
Suppose that $Y \sim \mathcal{N}(\theta, \sigma^2I_n)$, $\theta \in \ell_0[p_n]$ and $\tilde p_n = \#\{i: \theta_i \neq 0\}$. Note that $\tilde p_n \leq p_n$. Assume without loss of generality that  for $i = 1, \ldots, \tilde p_n$, $\theta_i \neq 0$, while for $i = \tilde p_n + 1, \ldots, n$, $\theta_i = 0$. We split up the expectation $\E_\theta \|T_\tau(Y) - \theta\|^2 $ into the two corresponding parts:
\begin{align*}
\sum_{i=1}^n \E_{\theta_i}(T_\tau(Y_i) - \theta_i)^2 &= \sum_{i=1}^{\tilde p_n} \E_{\theta_i}(T_\tau(Y_i) - \theta_i)^2 + \sum_{i=\tilde p_n+1}^n \E_0 T_\tau(Y_i)^2.
\end{align*}
We will now show that these two terms can be bounded by $\tilde p_n(1 + \log \frac{1}{\tau})$  and  $(n-\tilde p_n)\sqrt{\log(1/\tau)}\tau$ respectively, up to multiplicative constants only depending on $\sigma$, for any choice of $\tau$ such that $\tau \in (0,1)$. 

\textit{Nonzero parameters}\\
Denote $\zt = \sqrt{2\sigma^{2} \log (1/\tau)}$.  We will show 
\begin{equation}\label{eq:goalmsenonzero}
\E_{\theta_i}(T_\tau(Y_i) - \theta_i)^2 \lesssim \sigma^2 + \zt^2.
\end{equation}
for all nonzero $\theta_i$, which can be done by bounding $\sup_y |T_\tau(y) - y|$: 
\begin{align*}
\E_{\theta_i}(T_\tau(Y_i) - \theta_i)^2 &= \E_{\theta_i}( (T_\tau(Y_i) - Y_i) + (Y_i - \theta_i))^2 \\
&\leq 2\E_{\theta_i}(Y_i - \theta_i)^2 + 2\E_{\theta_i}(T_\tau(Y_i) - Y_i)^2\\
&\leq 2\sigma^2 +  2\left(\sup_y |T_\tau(y) - y|\right)^2,
\end{align*} 
Lemma \ref{lem:diffbound} yields the following bound on the difference between the observation and the horseshoe estimator:
$|T_\tau(y) - y| \leq h(y, \tau)$,
where $h(y, \tau)$ is such that 
$\lim_{\tau \downarrow 0} \sup_{|y| > c\zt} h(y, \tau) = 0$
for any $c > 1$. Combining this with the inequality $|T_\tau(y) - y| \leq |y|$, we have as $\tau \to 0$:
\begin{equation}\label{eq:maxdiff}
\arg\max_y |T_\tau(y) - y| \lesssim \zt,
\end{equation}
which implies (\ref{eq:goalmsenonzero}), as $|T_\tau(y)| \leq |y|$:
\begin{equation*}
\left(\sup_y |T_\tau(y) - y|\right)^2 \lesssim \zt^2.
\end{equation*}

\textit{Parameters equal to zero}\\
We split up the term for the zero means into two parts:
\begin{equation*}
\E_0 T_\tau(Y)^2 = \E_0 T_\tau(Y)^2\1_{|Y| \leq \zt} + \E_0 T_\tau(Y)^2\1_{|Y| > \zt},
\end{equation*}
where $\zt = \sqrt{2\sigma^2\log(1/\tau)}$.  For the first term, we have, by the first bound in Lemma \ref{lem:postmeanupperbounds}:
\begin{align*}
\E_0T_\tau(Y)^2\1_{\{|Y| \leq \zt\}}
&= \int_{-\zt}^{\zt} T_\tau(y)^2 \frac{1}{\sqrt{2\pi\sigma^2}}e^{-\frac{y^2}{2\sigma^2}}dy\\
&\leq \int_{-\zt}^{\zt}  y^2e^{\frac{y^2}{\sigma^2}}f(\tau)^2  \frac{1}{\sqrt{2\pi\sigma^2}}e^{-\frac{y^2}{2\sigma^2}}dy
= \frac{f(\tau)^2}{\sqrt{2\pi\sigma^2} } \int_{-\zt}^{\zt}   y^2e^{\frac{y^2}{2\sigma^2}}dy\\
&\leq  \sqrt{\frac{2}{\pi}}\sigma f(\tau)^2 \zt \frac{1}{\tau}
\leq \sqrt{\frac{2}{\pi}}\sigma \frac{4}{9}\zt\tau \lesssim \zt \tau,
\end{align*}
where the identity $\frac{d}{dy} y e^\frac{y^2}{2\sigma^2} = \frac{y^2}{\sigma^2}e^\frac{y^2}{2\sigma^2} + e^\frac{y^2}{2\sigma^2}$ was used to bound $\int_{-\zt}^{\zt}   y^2e^{\frac{y^2}{2\sigma^2}}dy$. For the second term, because $|T_\tau(y)| \leq |y|$ for all $y$, we have by the identity $y^2\phi(y) = \phi(y) - \frac{d}{dy}[y\phi(y)]$, and by Mills' ratio:
\begin{align*}
\E_0 T_\tau(Y)^2\1_{\{|Y| > \zt\}} &\leq \E_0 Y^2\1_{\{|Y| > \zt\}} = 2\int_\frac{\zt}{\sigma}^\infty \sigma^2y^2 \phi(y) dy\\
&\leq 2\sigma\zt\phi\left(\frac{\zt}{\sigma}\right) + 2\sigma^3\frac{\phi\left(\frac{\zt}{\sigma}\right)}{\zt}
\leq 4\sigma \zt\phi\left(\frac{\zt}{\sigma}\right) = 4\sigma\zt\frac{1}{\sqrt{2\pi}}\tau,
\end{align*}
where the last inequality holds for $\zt > \sigma^{2}$. If we apply this inequality and combine this upper bound with the upper bound on the first term, we find, for $\zt > \sigma^{2}$ (corresponding to $\tau < e^{-\frac{\sigma^{2}}{2}}$): 
\begin{equation}\label{eq:goalmsezero}
\E_0 T_\tau(Y)^2 = \E_0 T_\tau(Y)^2\1_{\{|Y| \leq \zt\}} + \E_0 T_\tau(Y)^2\1_{\{|Y| > \zt\}} \lesssim \zt\tau.
\end{equation}
Hence, for $\tau <  e^{-\frac{\sigma^{2}}{2}}$: 
\begin{equation} \label{eq:boundzero}
 \sum_{i=p_n+1}^n \E_0 T_\tau(Y_i)^2  \lesssim (n-p_n)\zt\tau.\\
\end{equation}

\textit{Conclusion}\\
By (\ref{eq:goalmsenonzero}) and (\ref{eq:boundzero}), we find for $\tau <  e^{-\frac{\sigma^{2}}{2}}$:
\begin{align*}
\sum_{i=1}^n \E_{\theta_i}(T_\tau(Y_i) - \theta_i)^2 
&\lesssim \tilde p_n(1 +  \zt^2) + (n-\tilde p_n)\tau \zt. 
\end{align*} 
\end{proof}

\begin{lem} \label{lem:postvar}
The posterior variance when using the horseshoe prior can be expressed as:
\begin{equation} \label{eq:postvar}
\var(\theta \mid y) = \frac{\sigma^2}{y}T_\tau(y) - \left(T_\tau(y) - y\right)^2 
 + y^2 \frac{\displaystyle \int_0^1 (1-z)^2z^{-\frac{1}{2}}\frac{1}{\tau^2 + (1-\tau^2)z }e^{\frac{y^2}{2\sigma^2}z} dz  } {\displaystyle\int_0^1 z^{-\frac{1}{2}}\frac{1}{\tau^2 + (1- \tau^2)z}e^{\frac{y^2}{2\sigma^2}z} dz   },
\end{equation}
and bounded from above by:
\begin{enumerate}
\item $\var(\theta \mid y) \leq \sigma^2 + y^2$;
\item $\var(\theta \mid y) \leq \left(\frac{\sigma^2}{y} + y\right)T_\tau(y) - T_\tau(y)^2$.
\end{enumerate}
\end{lem}

\begin{proof}
As proven in \cite{Pericchi1992}:
\begin{equation*}
\var(\theta \mid y) = \sigma^2 + \sigma^4 \frac{d^2}{d y^2} \log m(y) = \sigma^2 - \left(\sigma^2 \frac{m'(y)}{m(y)} \right)^2 + \sigma^4 \frac{m''(y)}{m(y)},
\end{equation*}
where $m(y)$ is the density of the marginal distribution of $y$. Equality (\ref{eq:postvar}) can be found by combining the expressions
\begin{align*}
m(y) &= \frac{1}{\sqrt{2\pi^3}\sigma\tau}e^{-\frac{y^2}{2\sigma^2}}\int_0^1 z^{-\frac{1}{2}} \frac{1}{1 - \left(1 - \frac{1^2}{\tau^2}\right)z}e^{\frac{y^2}{2\sigma^2}z}dz\\
m''(y) &= \frac{1}{y} m'(y) + \frac{1}{\sqrt{2\pi^3}\sigma\tau} \frac{y^2}{\sigma^4}e^{-\frac{y^2}{2\sigma^2}}\int_0^1 z^{-\frac{1}{2}}(1-z)^2 \frac{1}{1 - \left(1 - \frac{1}{\tau^2}\right)z}e^{\frac{y^2}{2\sigma^2}z}dz
\end{align*}
with the equality $T_\tau(y) = y + \sigma^2 \frac{m'(y)}{m(y)}$. The first upper bound is implied by the property $|T_\tau(y)| < |y|$ and the fact that $(1-z)^2 \leq 1$ for $z \in [0,1]$. The second upper bound can be demonstrated by noting that $(1-z)^2 \leq 1 - z$ for $z \in [0,1]$ and hence:
\begin{equation*}
\var(\theta \mid y) \leq \frac{\sigma^2}{y} T_\tau(y) - (y - T_\tau(y))^2 + y^2\left(1 - \frac{1}{y}T_\tau(y)\right).
\end{equation*}
\end{proof}

\textbf{Proof of Theorem \ref{thm:var}}
\begin{proof}
As in the proof of Theorem \ref{thm:mse} we assume that $\theta_{i} \neq 0$ for $i = 1, \ldots, \tilde p_n$ and $\theta_{i} = 0$ for $i = \tilde p_n+1, \ldots, n$, where $\tilde p_n \leq p_n$ by assumption. We consider the posterior variances for the zero and nonzero means separately. Denote $\zt = \sqrt{2\sigma^{2}\log(1/\tau)}$.\\
\\
\textit{Nonzero means}\\
By applying the same reasoning as in Lemma \ref{lem:diffbound} to the final term of $\var(\theta|y)$ in (\ref{eq:postvar}), we can find a function $\tilde h(y, t)$ such that 
$\var(\theta | y) \leq \tilde h(y, \tau)$, where $\tilde h(y, \tau) \to \sigma^{2}$ as $y \to \infty$ for any fixed $\tau$. If $\tau \to 0$, the function $\tilde h(y, \tau)$ displays the following limiting behaviour for any $c > 1$:
\begin{equation*}
\lim_{\tau \downarrow 0} \sup_{|y| > c\zt} \tilde h(y, \tau) = \sigma^2.
\end{equation*}
Hence, as $\tau \to 0$: $\var(\theta | y) \lesssim \sigma^{2}$, for any $|y|$ that increases as least as fast as $\zt$ when $\tau$ decreases. Now suppose $|y| \leq \zt$. Then, by the bound $\var(\theta \mid y) \leq \sigma^{2} + y^{2}$ from Lemma \ref{lem:postvar}, we find:
\begin{equation*}
\var(\theta \mid y) \leq \sigma^{2} + \zt^{2}.
\end{equation*}
Therefore:
\begin{equation}\label{eq:varnonzero}
\sum_{i=1}^{\tilde p_n} \E_{\theta_{i}} \var(\theta_{i} \mid Y_{i}) \lesssim \tilde p_n(1 + \zt^{2}).
\end{equation}
\textit{Zero means}\\
By the bound $\var(\theta \mid y) \leq \sigma^{2} + y^{2}$, we find for $c \geq 1$:
\begin{align*} 
\E_{0} \var(\theta\mid Y)\1_{\{|Y| > c\zt\}}
&\leq 2\int_{c\zt}^{\infty}(\sigma^{2} + y^{2})\frac{1}{\sqrt{2\pi\sigma^{2}}}e^{-\frac{y^{2}}{2\sigma^{2}}}dy\\
&= 2\sigma^{2}\Phi^{c}\left(\frac{c\zt}{\sigma} \right) + 2\int_{\frac{c\zt}{\sigma}}^{\infty} \sigma^{2}x^{2}\phi(x)dx\\
&\leq 4\sigma^{3} \frac{\phi\left( \frac{c\zt}{\sigma} \right) }{c\zt} + 2\sigma c\zt \phi\left( \frac{c\zt}{\sigma} \right)  
\lesssim \frac{\tau}{\zt} + \zt\tau.
\end{align*}
 For $|y| < c\zt$, we consider the upper bound $\var(\theta \mid y) \leq \left(\frac{\sigma^{2}}{y} + y\right)T_\tau(y) - T_\tau(y)^2$ from Lemma \ref{lem:postvar}. From this bound, we get $\var(\theta \mid y) \leq \frac{\sigma^2}{y}T_\tau(y) + yT_\tau(y)$.  Hence:
 \begin{align}\label{eq:postvarboundintermediate}
\E_{0} \var(\theta \mid Y)\1_{\{|Y| \leq c\zt\}}  &\leq
\sigma^2 \int_{-c\zt}^{c\zt} \frac{1}{y}T_\tau(y)\frac{1}{\sqrt{2\pi\sigma^{2}}} e^{-\frac{y^{2}}{2\sigma^{2}}}dy\notag\\ 
&\quad + \int_{-c\zt}^{c\zt} yT_\tau(y)\frac{1}{\sqrt{2\pi\sigma^{2}}} e^{-\frac{y^{2}}{2\sigma^{2}}}dy.
\end{align}
We bound the first integral from (\ref{eq:postvarboundintermediate}) by applying the first bound on $T_\tau(y)$ from Lemma \ref{lem:postmeanupperbounds}:
 \begin{align*}
\sigma^2 \int_{-c\zt}^{c\zt} \frac{1}{y}T_\tau(y)\frac{1}{\sqrt{2\pi\sigma^{2}}} e^{-\frac{y^{2}}{2\sigma^{2}}}dy
&\leq \sigma^2\int_{-c\zt}^{c\zt} f(\tau)\frac{1}{\sqrt{2\pi\sigma^{2}}} dy\\
&= \sqrt{\frac{2\sigma}{\pi}}c\zt f(\tau) \lesssim \zt\tau,
\end{align*}
because $f(\tau) \leq \frac{2}{3}\tau$. For the second term in (\ref{eq:postvarboundintermediate}), we first note that the second bound from Lemma \ref{lem:postmeanupperbounds} can be relaxed to:
\begin{equation}\label{eq:pmboundrelaxed}
T_\tau(y) \leq \tau y\left(\frac{2}{3}\tau e^{\tau^2 \frac{y^2}{2\sigma^2}} + \frac{2}{\sqrt{a}} e^{\frac{y^2}{2a\sigma^2}} + 2\sqrt{a}\sigma^2\frac{1}{y^2}e^{\frac{y^2}{2\sigma^2}} \right)
\end{equation}
for any $a > 1$ and $\tau < \frac{1}{a}$. By plugging this bound into the second integral of (\ref{eq:postvarboundintermediate}), we get three terms, which we will name $I_1, I_2$ and $I_3$ respectively. We then find, bounding above by the integral over $\mathbb{R}$ instead of $[-c\zt,c\zt]$ for $I_1$ and $I_2$:
\begin{align*} 
I_1 &= \frac{2}{3}\tau^2 \int_{-c\zt}^{c\zt} y^2 \frac{1}{\sqrt{2\pi\sigma^2}}e^{-(1-\tau^2)\frac{y^2}{2\sigma^2}}dy
\leq \frac{2}{3}\tau^2 \frac{\sigma^2}{(1-\tau^2)^\frac{3}{2}} \lesssim \tau^2.\\
I_2 &= \frac{2}{\sqrt{a}} \tau \int_{-c\zt}^{c\zt} y^2 \frac{1}{\sqrt{2\pi\sigma^2}} e^{-\frac{a-1}{a}\frac{y^2}{2\sigma^2}}dy \leq \frac{2a\sigma^2}{(a-1)^\frac{3}{2} }\tau \lesssim \tau.\\
I_3 &= 2\sqrt{a}\sigma^2\tau \int_{-c\zt}^{c\zt} \frac{1}{\sqrt{2\pi\sigma^2}}dy = \frac{2\sqrt{2a}c\sigma}{\sqrt{\pi}}\zt\tau \lesssim \zt\tau.
\end{align*}
And thus:
\begin{equation}\label{eq:varzero}
\sum_{i = \tilde p_n+1}^{n} \E_{0} \var(\theta_{i} \mid Y_{i}) \lesssim (n-\tilde p_n)\left(\zt + \tau + 1 \right)\tau.
\end{equation}
\textit{Conclusion}\\
By (\ref{eq:varnonzero}) and (\ref{eq:varzero}):
\begin{equation*}
\E_{\theta} \sum_{i = 1}^{n} \var(\theta_{i} \mid Y_{i}) \lesssim \tilde p_n(1 + \zt^{2}) + (n-\tilde p_n)\left(\zt + \tau + 1\right)\tau.
\end{equation*}
\end{proof}

\textbf{Proof of Theorem \ref{thm:varlowerbound}}
\begin{proof}
 By expanding $(1-z)^2z^{-\frac{1}{2}} = z^{-\frac{1}{2}} - 2z^{\frac{1}{2}} + z^{\frac{3}{2}}$, we see that the final term in \eqref{eq:postvar} is equal to:
\begin{equation*}
y^2 - 2yT_\tau(y) + y^2 \frac{\ \int_0^1 z^{\frac{3}{2}}\frac{1}{\tau^2 + (1-\tau^2)z }e^{\frac{y^2}{2\sigma^2}z} dz  } {\int_0^1 z^{-\frac{1}{2}}\frac{1}{\tau^2 + (1- \tau^2)z}e^{\frac{y^2}{2\sigma^2}z} dz   }.
\end{equation*} 
As $\frac{T_\tau(y)}{y}$ is non-negative, we can bound the posterior variance from below by the final two terms in (\ref{eq:postvar}). By the above equality, this yields the following lower bound:
\begin{equation*}
\var(\theta\mid y) \geq y^2 \frac{I_{\frac{3}{2}}(y)}{I_{-\frac{1}{2}}(y)} - T_\tau(y)^2 = y^2\left(\frac{I_{\frac{3}{2}}(y)}{I_{-\frac{1}{2}}(y)} - \left(\frac{I_{\frac{1}{2}}(y)}{I_{-\frac{1}{2}}(y)} \right)^2 \right),
\end{equation*}
where $I_k$ is as in Lemma \ref{lem:integralbounds}. We now use the bounds from Lemma \ref{lem:integralbounds} with $a = 2$ and take $\xi$ equal to $c\log(1/\tau)$ for some nonnegative constant $c$. Then $e^\xi = \frac{1}{\tau^c}$ and $e^\frac{\xi}{2} = \frac{1}{\tau^\frac{c}{2}}$. Taking for each bound on $I_k$, $k \in \left\{\tfrac{3}{2}, \tfrac{1}{2}, -\tfrac{1}{2}\right\}$, the term that diverges fastest as $\tau$ approaches zero, we find that the lower bound is asymptotically of the order:
\begin{equation*}
2\sigma^2\xi \left( \frac
{\frac{1}{2\sqrt{2}\xi} \frac{1}{\tau^c}}
{\max\left\{ \frac{2e^{\tau \xi}}{\tau}, \frac{2\sqrt{2}}{\xi} \frac{1}{\tau^c}  \right\} }
- \left( 
\frac{ \frac{\sqrt{2}}{\xi} \frac{1}{\tau^c} }
{\max\{ \frac{e^{\tau^2\xi}}{\tau}, \frac{1}{2\xi}\frac{1}{\tau^c} \}}
\right)^2 \right). 
\end{equation*}
For $c \leq 1$, this reduces to:
\begin{equation*}
\frac{\sigma^2}{2\sqrt{2}} e^{-\tau\xi}\tau^{1-c} - \frac{4\sigma^2}{\xi}e^{-2\tau^2\xi}\tau^{2-2c}.
\end{equation*}
The second term is negligible compared to the first. Hence, we will use the term $\frac{\sigma^2}{2\sqrt{2}} e^{-\tau\xi}\tau^{1-c}$ as our lower bound on $\var(\theta \mid y)$ for $y = \pm \sqrt{2c\sigma^2 \log(1/\tau)} = \sqrt{c}\zt$, where $\zt = \sqrt{2\sigma^2\log(1/\tau)}$. To find the lower bound on $\sum_{i=1}^n \E_{\theta_i} \var(\theta_i \mid Y_i)$, we only need to consider the parameters equal to zero:
\begin{equation} \label{eq:varlowerboundzero}
\sum_{i=1}^n \E_{\theta_i} \var(\theta_i \mid Y_i) \geq (n-p_n) \E_0\var(\theta_i \mid Y_i)\1_{\left\{ |Y_i| \leq \zt \right\} }.
\end{equation}
By the substitution $x = y^2/\zt^2, dy = \frac{ \sigma \sqrt{\log(1/\tau)}}{\sqrt{2x}}dx$, we find:
\begin{align}
E_0\var(\theta_i \mid Y_i)\1_{\{ |Y_i| \leq \zt  \} } 
&\geq
2 \int_0^{\zt } \frac{\sigma^2}{2\sqrt{2}} e^{-\tau \frac{y^2}{2\sigma^2}} \tau^{1 - \frac{y^2}{\zt^2}} \frac{1}{\sqrt{2\pi\sigma^2}} e^{-\frac{y^2}{2\sigma^2}} dy \notag\\
&= \frac{\sigma}{4\sqrt{\pi}} \tau \zt   \int_0^1 \frac{\tau^{\tau x}}{\sqrt{x}}dx
\geq \frac{\sigma}{2\sqrt{\pi}} e^{-\frac{1}{e}} \tau \zt, \label{eq:Thm34Tau}
\end{align}
where in the last step, we used $\tau^{\tau x} \geq \tau^\tau \geq e^{-\frac{1}{e}} $ for $x \in [0, 1], \tau \in (0, 1]$. By plugging this into (\ref{eq:varlowerboundzero}), we find that as $\tau \to 0$:
\begin{equation}\label{eq:lowerboundvarres}
\sum_{i=1}^n \E_{\theta_i} \var(\theta_i \mid Y_i) \gtrsim (n-p_n)\tau\zt,
\end{equation}
finishing the proof for the first statement of the theorem. 

We now consider $\theta$ such that $\theta_i = a_n$ for $i = 1, \ldots, p_n$, and $\theta_i = 0$ for $i = p_n + 1, \ldots, n$, and assume without loss of generality that $a_n > 0$. We wish to find  conditions on $a_n$ such that the lower bound \eqref{eq:lowerboundvarres} is sharp (up to a  constant factor). Denoting  $\zt = \sqrt{2\sigma^2\log(1/\tau)}$, as before, it is sufficient if we can find $a_n$ such that 
$
 \E_{\theta_i = a_n} \var(\theta_i \mid Y_i) \lesssim \tau\zt,
$
because in combination with the bound \eqref{eq:varzero}, this will yield $\sum_{i=1}^n \E_{\theta_i} \var(\theta_i \mid Y_i) \lesssim n\tau\zt$, which is of the same order as \eqref{eq:lowerboundvarres}, as $p_n = o(n)$. Sufficient conditions on $a_n$ can be found by adapting the proof for the `zero means' case of Theorem \ref{thm:var}. 

We first consider $|y_i| > \zt$. By the first bound of Lemma \ref{lem:postvar}:
\begin{align}\label{geval1start}
\E_{\theta_i} \var(\theta_i \mid Y_i)\1_{\{|Y_i| > \zt\}} 
&\leq \int_{\zt}^\infty (\sigma^2 + y^2) \frac{1}{\sqrt{2\pi\sigma^2}} e^{-\frac{(y-a_n)^2}{2\sigma^2}}dy \notag\\
&\quad + \int_{-\infty}^{-{\zt}} (\sigma^2 + y^2) \frac{1}{\sqrt{2\pi\sigma^2}} e^{-\frac{(y-a_n)^2}{2\sigma^2}}dy.
\end{align}
The first integral from \eqref{geval1start} can be split into two parts by splitting up the factor $\sigma^2 + y^2$, the first of which can be bounded, by substituting  $x = (y-a_n)/\sigma$ and applying Mills' ratio:\begin{equation} \label{geval1grens1}
\sigma^2\int_{(\zt-a_n)/\sigma}^\infty \phi(x) dx = \sigma^2 \Phi^c\left(\frac{\zt-a_n}{\sigma}\right) 
\leq \frac{\sigma^3}{\zt - a_n}\phi\left(\frac{\zt-a_n}{\sigma}\right).
\end{equation}
The second of these integrals is, by $y^2 = (y-a_n)^2 - a_n^2 + 2a_ny$, equal to:
\begin{align} \label{geval1stapanders}
\int_{\zt}^\infty (y-a_n)^2&\frac{1}{\sqrt{2\pi\sigma^2}} e^{-\frac{(y-a_n)^2}{2\sigma^2}}dy 
- a_n^2\int_{\zt}^\infty \frac{1}{\sqrt{2\pi\sigma^2}} e^{-\frac{(y-a_n)^2}{2\sigma^2}}dy\notag\\
&+ a_n \int_{\zt}^\infty y\frac{1}{\sqrt{2\pi\sigma^2}} e^{-\frac{(y-a_n)^2}{2\sigma^2}}dy.
\end{align}
The second integral of \eqref{geval1stapanders} can be bounded from below by zero, and the third from above by $a_n\E_{\theta_i}Y_i = a_n^2$. Again substituting $x = (y-a_n)/\sigma$ yields the following upper bound on \eqref{geval1stapanders}: $\sigma^2\int_{(\zt-a_n)/\sigma}^\infty x^2\phi(x)dx + a_n^2$.
Now using the equality $x^2\phi(x) = \phi(x) - \frac{d}{dx}[x\phi(x)]$ and again Mills' ratio, and combining  with \eqref{geval1grens1}, we find the following upper bound on the first integral from \eqref{geval1start}:
\begin{equation} \label{geval1grens3}
\frac{2\sigma^3}{\zt - a_n}\phi\left(\frac{\zt-a_n}{\sigma}\right) + \sigma(\zt-a_n)\phi\left(\frac{\zt-a_n}{\sigma}\right) + a_n^2.
\end{equation}
By substituting $x = -y$ in the second integral from \eqref{geval1start} and then applying the same inequalities to it as to the first integral, the following bound is obtained:
\begin{equation} \label{geval1grens4}
\frac{2\sigma^3}{\zt+a_n}\phi\left(\frac{\zt+a_n}{\sigma} \right) + \sigma(\zt + a_n)\phi\left(\frac{\zt+a_n}{\sigma} \right).
\end{equation}
This bound does not include a term $a_n^2$, because in the step equivalent to \eqref{geval1stapanders}, the identity $y^2 = (y+a_n)^2 - a_n^2 - 2ya_n$ is used, and thus only the integral 
$\int_{\zt}^\infty (y+a_n)^2\frac{1}{\sqrt{2\pi\sigma^2}} e^{-\frac{(y+a_n)^2}{2\sigma^2}}dy $
needs to be bounded in that step. $\E_{\theta_i} \var(\theta\mid Y)\1_{\{|Y| > \zt\}}$ can thus be bounded by the sum of \eqref{geval1grens3} and \eqref{geval1grens4}.  The factor $\phi((\zt+a_n)/\sigma)$ can be bounded from above by $\phi(\zt/\sigma) = \tau/\sqrt{2\pi}$. The factor $\phi((\zt-a_n)/\sigma)$ is equal to $ \frac{1}{\sqrt{2\pi}} e^{-\frac{\zt^2}{2\sigma^2}}e^{-\frac{a_n^2}{2\sigma^2}}e^{\frac{\zt a_n}{\sigma}} 
= \frac{\tau}{\sqrt{2\pi}} e^{-\frac{a_n^2}{2\sigma^2}}e^{\frac{\zt a_n}{\sigma}}$. 
Hence we arrive at the following upper bound:
\begin{equation}\label{geval1resultaat}
\frac{\sigma}{\sqrt{2\pi}}\left[ 
\left(\frac{2\sigma^2}{\zt - a_n} + \zt - a_n\right) e^{-\frac{a_n^2}{2\sigma^2}}e^{\frac{\zt a_n}{\sigma}}
+ \frac{2\sigma^2}{\zt + a_n} + \zt + a_n
\right] \tau
+ a_n^2.
\end{equation}
If $a_n \lesssim1/\zt$, then $ e^{-\frac{a_n^2}{2\sigma^2}}e^{\frac{\zt a_n}{\sigma}} = \mathcal{O}(1)$ and $\zt \pm a_n = \mathcal{O}(\zt)$, yielding an upper bound on \eqref{geval1resultaat} of order $\tau\zt$. 

We now consider $|y_i| \leq \zt$. We use the second bound of Lemma \ref{lem:postvar}:
\begin{align}\label{geval2start}
\E_{\theta_i} \var(\theta_i \mid Y_i)\1_{\{|Y_i| \leq \zt\}} 
&\leq
\sigma^2 \int_{-\zt}^{\zt} \frac{1}{y}T_\tau(y)\frac{1}{\sqrt{2\pi\sigma^2}} e^{- \frac{(y-a_n)^2}{2\sigma^2}}dy \notag\\
&\quad +
\sigma^2 \int_{-\zt}^{\zt} yT_\tau(y)\frac{1}{\sqrt{2\pi\sigma^2}} e^{- \frac{(y-a_n)^2}{2\sigma^2}}dy.
\end{align} 
Applying inequality $\frac{1}{y}T_\tau(y) \leq \frac{2}{3}\tau e^{\frac{y^2}{2\sigma^2}}$ from Lemma \ref{lem:postmeanupperbounds} to the first integral yields the bound:
\begin{align*}
\frac{\sqrt{2}\sigma}{3\sqrt{\pi}}\tau \int_{-\zt}^{\zt} e^{\frac{y^2}{2\sigma^2}} e^{- \frac{(y-a_n)^2}{2\sigma^2}} dy 
&=
\frac{\sqrt{2}\sigma}{3\sqrt{\pi}}\tau e^{-\frac{a_n^2}{2\sigma^2}}  \int_{-\zt}^{\zt} e^{\frac{a_ny}{\sigma^2}}dy
\leq \frac{\sqrt{2}\sigma}{3\sqrt{\pi}}\tau e^{-\frac{a_n^2}{2\sigma^2}} 2\zt e^{\frac{a_n\zt}{\sigma^2}}.
\end{align*}
If $a_n \lesssim 1/\zt$, we have $a_n\zt = \mathcal{O}(1)$ and thus this term will be of order $\tau\zt$. For the second integral from \eqref{geval2start}, we use  bound \eqref{eq:pmboundrelaxed}. This leads to three integrals to be bounded, $I_1, I_2$ en $I_3$. 
\begin{align*}
I_1 
&= \frac{\sigma}{\sqrt{2\pi}}\frac{2}{3}\tau^2 e^{\frac{\tau^2}{1-\tau^2}\frac{a_n^2}{2\sigma^2}} \int_{-\zt}^{\zt} y^2 e^{-\frac{1}{2\sigma^2/(1-\tau^2)}\left(y - \frac{a_n}{1-\tau^2}\right)^2 }dy\\
&\leq  \frac{2}{3}e^{\frac{\tau^2}{1-\tau^2}\frac{a_n^2}{2\sigma^2}} \frac{\sigma^2}{(1-\tau^2)^{3/2}}\left(\sigma^2 + \frac{a_n^2}{1-\tau^2}\right)\tau^2.\\
I_2 
&= \frac{2\sigma}{\sqrt{a}\sqrt{2\pi}}\tau e^{\frac{a_n^2}{(b-1)2\sigma^2}} \int_{-\zt}^{\zt} y^2 e^{-\frac{1}{2\sigma^2 \frac{a}{a-1}} \left(y - \frac{a}{a-1}a_n\right)^2}\\
&\leq  \frac{2}{\sqrt{b}} e^{\frac{a_n^2}{(b-1)2\sigma^2}} \sigma^2 \left(\frac{a}{a-1}\right)^{3/2}\left(\sigma^2 + \frac{a}{a-1}a_n^2\right)\tau.\\
I_3 &= \frac{2\sqrt{a}\sigma^3}{\sqrt{2\pi}}\tau \int_{-\zt}^{\zt} e^{\frac{y^2}{2\sigma^2}} e^{-\frac{(y-a_n)^2}{2\sigma^2}} dy 
\leq \frac{2\sqrt{2a}\sigma^3}{\sqrt{\pi}}e^{-\frac{a_n^2}{2\sigma^2}}e^{\frac{a_n\zt}{\sigma^2}}\tau\zt.
\end{align*}
$I_1, I_2$ and $I_3$ will all be of no larger order than $\tau\zt$ if $a_n \lesssim 1/\zt$.
\end{proof}

\begin{lem}\label{lem:mismatch1}
For all $k \in \mathbb{R}$, $\int_1^y u^ke^u du=y^ke^y(1+\mathcal{O}(1/y))$, as $y \to \infty$.
\end{lem}

\begin{proof}
For $k = 0$, the statement is immediate. By integration by parts the integral is seen to be equal to 
$y^ke^y-e-\int_1^y ku^{k-1}e^u du$.  For $k \neq 0$, the latter integral is bounded above by
\begin{equation*}
|k|\int_1^{y/2}(1\vee y/2)^{k-1}e^u du+|k|\int_{y/2}^y (y/2\vee y)^{k-1}e^u du.
\end{equation*}
This is further bounded above by a multiple of $(1\vee y^{k-1})e^{y/2}+y^{k-1}e^y$.
\end{proof}

\begin{lem}\label{lem:mismatch2}
Let $I_k$ be as in Lemma \ref{lem:integralbounds}. There exist functions $R_k$ with $\sup_{\zt/4\leq y\leq 4\zt}|R_k(y)| \to 0$ for $k > 0$ and $k = -\tfrac{1}{2}$, such that,
\begin{align*}
I_k(y) &= \left(\tau^{2k}\int_0^1\frac{z^k}{1+z} dz+\frac{2\sigma^2}{y^2}e^{\frac{y^2}{2\sigma^2}}\right)\left(1+R_k(y)\right), \quad \text{for } k > 0,\\
I_{-\frac{1}{2}}(y)&=\left(\tau^{-1}\int_0^\infty\frac{1}{\sqrt{z}(1+z)} dz+ \frac{2\sigma^2}{y^2}e^{\frac{y^2}{2\sigma^2}} \right)
\left(1+R_{-\frac{1}{2}}(y)\right).
\end{align*} 
\end{lem}

\begin{proof}
We split the integral in the definition of $I_k$ over the intervals $[0,\tau^2]$ and $[\tau^2,1]$. The first interval contributes, uniformly in $y\tau \to 0$,
\begin{align}
\int _0^{\tau^2} \frac{z^ke^{\frac{y^2}{2\sigma^2}z}}{\tau^2+(1-\tau^2)z} dz
&=\int _0^{\tau^2} \frac{z^k}{\tau^2+(1-\tau^2)z} dz\,(1+o(1))\notag\\
&=\tau^{2k} \int _0^{1} \frac{u^k}{1+(1-\tau^2)u} du\,(1+o(1)),\label{eq:lemA6int01}
\end{align}
by the substitution $u = z/\tau^2$. The integral tends to $\int _0^{1} \frac{u^k}{1+u}\,du$, by the dominated convergence theorem, for any $k>-1$.  The second interval contributes, with the substitution $u = (y^2/2\sigma^2)z$:
\begin{align*}
\int _{\tau^2}^1 \frac{z^ke^{\frac{y^2}{2\sigma^2}z}}{\tau^2+(1-\tau^2)z} dz
&= \left(\frac{2\sigma^2}{y^2}\right)^k \left(\int_{\frac{y^2}{2\sigma^2}\tau^2}^1+\int_1^{\frac{y^2}{2\sigma^2}}\right)
\frac{u^ke^u}{\frac{y^2}{2\sigma^2}\tau^2+(1-\tau^2)u} du.
\end{align*}
In the second integral the argument satisfies $u \geq 1$, and hence $u/((y^2\tau^2/(2\sigma^2)+(1-\tau^2)) \to 1$, uniformly in $u$ and $y\tau \to 0$. Hence 
\begin{align*}
 \left(\frac{2\sigma^2}{y^2}\right)^k \int_1^{\frac{y^2}{2\sigma^2}} \frac{u^ke^u}{\frac{y^2}{2\sigma^2}\tau^2+(1-\tau^2)u} du
&\asymp  \left(\frac{2\sigma^2}{y^2}\right)^k\int_1^\frac{y^2}{2\sigma^2}u^{k-1}e^u du \\
&\asymp \frac{2\sigma^2}{y^2}e^{\frac{y^2}{2\sigma^2}}(1+o(1))
\end{align*}
as $y \to \infty$, by Lemma \ref{lem:mismatch1}. For the first integral we separately consider the cases $k>0$ and $k=-1/2$. If $k>0$, then $\int_0^1 u^{k-1}e^u du$ converges, and hence, by the dominated convergence theorem,
uniformly in $y\tau \to 0$,
\begin{equation*}
 \left(\frac{2\sigma^2}{y^2}\right)^k \int_{\tau^2\frac{y^2}{2\sigma^2}}^1 \frac{u^ke^u}{\frac{y^2}{2\sigma^2}\tau^2+(1-\tau^2)u} du
\to  \left(\frac{2\sigma^2}{y^2}\right)^k \int_0^1 u^{k-1}e^u du.
\end{equation*}
If $k=-1/2$ , then we substitute $v = 2\sigma^2u/(\tau^2y^2)$ and rewrite the integral  as
\begin{equation*}
\left(\frac{2\sigma^2}{y^2}\right)^{-\frac{1}{2}}\int_1^{\frac{2\sigma^2}{\tau^2y^2}}\frac{v^{-\frac{1}{2}}e^{\frac{\tau^2y^2}{2\sigma^2}v}}{1+(1-\tau^2)v} \left(\frac{\tau^2y^2}{2\sigma^2}\right)^{-\frac{1}{2}}dv
= \frac{1}{\tau} \int_1^\infty \frac{v^{-1/2}}{1+v} dv (1+o(1)).
\end{equation*}
This combines with the integral \eqref{eq:lemA6int01}.
\end{proof}

\textbf{Proof of Theorem \ref{thm:mismatch}}
\begin{proof}
Denote $\zt = \sqrt{2\sigma^2\log(1/\tau)}$ and assume that $\theta_i = \gamma\zt$ for $i = 1, \ldots, p_n$ and $\theta_i = 0$ for $i = p_n+1, \ldots, n$. We prove \eqref{eq:mismatchbias} by proving that there exists a positive constant $c_1(\gamma)$ such that
\begin{equation}\label{eq:mismatchbiasgoal}
\E_{\theta = \gamma\zt} T_\tau(Y) = \tau^{(1-\gamma)^2}\zt^{2\gamma-2}c_1(\gamma)(1+o(1)).
\end{equation} 
If \eqref{eq:mismatchbiasgoal} holds, we have, by Jensen's inequality:
\begin{equation}\label{eq:misbiaslower}
\sum_{i=1}^{p_n} \E_{\theta_i}(T_\tau(Y_i) - \theta_i)^2 \geq p_n(\tau^{(1-\gamma)^2}\zt^{2\gamma-2}c_1(\gamma) - \gamma\zt)^2 \gtrsim p_n\zt^2,
\end{equation}
as $\tau \to 0$. In addition, we have $T_\tau(y) = yI_{\frac{1}{2}}(y)/I_{-\frac{1}{2}}(y)$. For $|y| = \sqrt{2\sigma^2c\log(1/\tau)}$, with $c > 1$, the lower bound \eqref{eq:lower1/2} on $I_\frac{1}{2}(y)$ behaves as $(\sigma^2/y^2)e^\frac{y^2}{2\sigma^2}$, while the upper bound \eqref{eq:upper-1/2} on $I_{-\frac{1}{2}}(y)$ behaves as $(2a\sqrt{a}\sigma^2/y^2)e^\frac{y^2}{2\sigma^2}$, as $\tau \to 0$. Therefore, for $|y| > \zt$, we have $T_\tau(y) \gtrsim y$. Thus, we can bound by:
\begin{align}
\sum_{i=p_n+1}^n \E_{\theta_i} T_\tau(Y_i)^2 
&\geq (n-p_n) \E_{\theta = 0} T_\tau(Y)^2\1_{\{|Y|>\zt\}} \gtrsim (n-p_n)\int_{\frac{\zt}{\sigma}}^\infty y^2 \phi(y) dy\notag \\ 
&= (n-p_n) \left(\int_{\frac{\zt}{\sigma}}^\infty \phi(y) dy + \frac{\zt}{\sigma}\phi\left(\frac{\zt}{\sigma}\right) \right)
\gtrsim (n-p_n) \zt \phi\left(\frac{\zt}{\sigma}\right)\notag \\ 
&=  (n-p_n)\frac{1}{\sqrt{2\pi}} \tau\zt. \label{eq:misbiaslowerzero}
\end{align}
By combining the lower bounds \eqref{eq:misbiaslower} and \eqref{eq:misbiaslowerzero} with the upper bound \eqref{eq:mseupperbound}, we arrive at \eqref{eq:mismatchbias}. For the posterior variance, we already have $\sum_{i=p_n+1}^n \var(\theta_i \mid Y_i) \asymp (n-p_n)\tau\zt$ by \eqref{eq:varzero} and \eqref{eq:Thm34Tau}. Expression \eqref{eq:mismatchvar} can therefore be proven by showing that there exists a positive constant $c_2(\gamma)$ such that:
\begin{equation} \label{eq:mismatchvargoal}
\E_{\theta = \gamma\zt} \var(\theta \mid Y) = \tau^{(1-\gamma)^2}\zt^{2\gamma-1}c_2(\gamma)(1+o(1)).
\end{equation}

\textit{Proof of \eqref{eq:mismatchbiasgoal}}\\
The expected value $\E_{\theta = \gamma \zt}T_\tau(Y)$ is equal to
\begin{equation}\label{eq:mismatchbiasintegrals}
\frac{1}{\sigma}\left(\int_{-\infty}^{-\frac{\zt}{2}}+\int_{-\frac{\zt}{2}}^{3\zt}+\int_{3\zt}^\infty\right)
(\zt+y) \frac{I_{\frac{1}{2}}(\zt + y)}{I_{-\frac{1}{2}}(\zt + y)}\phi\left(\frac{y + (1-\gamma)\zt}{\sigma}\right) dy.
\end{equation}
We shall show that the first and third integrals are negligible, while the second gives the approximation in \eqref{eq:mismatchbiasgoal}. On the domain of the second integral, we have $\zt/4 \leq \zt + y \leq 4\zt$, so we can apply Lemma \ref{lem:mismatch2} to see that this integral is asymptotic to
\begin{equation}\label{eq:biassecond}
\frac{1}{\sigma}\int_{-\frac{\zt}{2}}^{3\zt}(\zt + y)\frac{
c_2\tau^2(\zt+y)^2 + 2\sigma^2e^{\frac{y^2+2y\zt}{2\sigma^2}}
}
{
c_1(y+\zt)^2 + 2\sigma^2e^{\frac{y^2+2y\zt}{2\sigma^2}}
}\phi\left(\frac{y + (1-\gamma)\zt}{\sigma}\right) dy,
\end{equation}
where $c_1 = \int_0^\infty z^{-1/2}(1-z)^{-1}dz$ and $c_2 = \int_0^1 z^{1/2}(1-z)^{-1}dz$. On $[-\zt/2, 3\zt]$:
\begin{align*}
c_2\tau^2(\zt+y)^2\phi\left(\frac{y + (1-\gamma)\zt}{\sigma}\right) &\leq \frac{c_2}{\sqrt{2\pi}}\tau^2(4\zt)^3e^{\frac{(1/2-\gamma)^2\zt^2}{2\sigma^2}} \\
&= \frac{64c_2}{\sqrt{2\pi}}\zt^3\tau^{2-(1/2-\gamma)^2},
\end{align*}
so \eqref{eq:biassecond} is asymptotic to:
\begin{equation*}
\mathcal{O}(\tau) + \frac{2\sigma}{\sqrt{2\pi}}e^{-\frac{(1-\gamma)^2\zt^2}{2\sigma^2}}\int_{-\frac{\zt}{2}}^{3\zt} \frac{(\zt+y)e^{\frac{\gamma\zt y}{\sigma^2}} }{c_1(y+\zt)^2 + 2\sigma^2e^{\frac{y^2+2y\zt}{2\sigma^2}}} dy.
\end{equation*}
By the substitution $u = \zt y - 2\sigma^2\log\zt$, the remaining integral is equal to, with $a_\tau =-\frac{\zt^2}{2} - 2\sigma^2\log\zt$ and $b_\tau = 3\zt^2 - 2\sigma^2\log\zt$:
\begin{align*}
\frac{2\sigma}{\sqrt{2\pi}}\tau^{(1-\gamma)^2}&\frac{1}{\zt} \int_{a_\tau}^{b_\tau}
 \frac{
(\zt + \zt^{-1}(u + 2\sigma^2\log\zt))e^\frac{\gamma u}{\sigma^2}\zt^{2\gamma}
}
{
c_1(\zt + \zt^{-1}(u + 2\sigma^2\log\zt))^2 + 2\sigma^2 e^\frac{u}{\sigma^2}\zt^2 e\frac{(u+2\sigma^2\log\zt)^2}{2\sigma^2\zt^2}
} du\\
&\sim \frac{2\sigma}{\sqrt{2\pi}}\tau^{(1-\gamma)^2}\frac{1}{\zt}\int_{-\infty}^\infty \frac{\zt e^\frac{\gamma u}{\sigma^2}\zt^{2\gamma}}{(c_1 + 2\sigma^2e^\frac{u}{\sigma^2})\zt^2}du,
\end{align*}
by the dominated convergence theorem. This yields the approximation in \eqref{eq:mismatchbiasgoal}, with $c_1(\gamma) = (2\sigma/\sqrt{2\pi})\int_{-\infty}^\infty e^\frac{\gamma u}{\sigma^2}/(c_1 + 2\sigma^2 e^\frac{u}{\sigma^2})du$.

For the first integral in \eqref{eq:mismatchbiasintegrals}, we use bound 1 from Lemma \ref{lem:postmeanupperbounds}, and obtain a bound on its absolute value equal to
\begin{align}
\frac{1}{\sigma} \int_{-\infty}^{-\frac{\zt}{2}} & |\zt + y| \tau e^\frac{ (\zt+y)^2}{2\sigma^2} \phi\left(\frac{y + (1-\gamma)\zt}{\sigma}\right) dy  \notag\\
&= \frac{2}{3\sqrt{2\pi}\sigma}\tau^{(1-\gamma)^2} \int_{-\infty}^{-\frac{\zt}{2}} |\zt + y| e^\frac{\gamma\zt y}{\sigma^2}dy
\lesssim \tau^{(1-\gamma)^2}e^{-\frac{\gamma\zt^2}{2\sigma^2}} = \tau^{(1-\gamma)^2 + \gamma}, \label{eq:mismatchbiasint1}
\end{align}
where the last inequality follows by integration by parts. This is of much smaller order than the second integral from \eqref{eq:mismatchbiasintegrals}. In the third integral of \eqref{eq:mismatchbiasintegrals}, we bound $I_\frac{1}{2}(\zt + y)/I_{-\frac{1}{2}}(\zt + y)$ by 1, giving the upper bound
\begin{equation*}
\frac{1}{\sigma}\int_{3\zt}^\infty (\zt +y) \phi\left(\frac{y+(1-\gamma)\zt}{\sigma}\right)dy \lesssim \phi\left(\frac{3\zt + (1-\gamma)\zt}{\sigma}\right) = \frac{1}{\sqrt{2\pi}}\tau^{4-\gamma},
\end{equation*}
by Mills' ratio. This is also of much smaller order than the second integral from \eqref{eq:mismatchbiasintegrals}, thus concluding the proof of \eqref{eq:mismatchbiasgoal}.

\textit{Proof of \eqref{eq:mismatchvargoal}}\\
By expanding the term $(1-z)^2$ in the numerator of the final term of \eqref{eq:postvar}, the posterior variance can be seen to be equal to:
\begin{equation}
\var(\theta \mid y) = \sigma^2\frac{I_\frac{1}{2}(y)}{I_{-\frac{1}{2}}(y)} + y^2\left[\frac{I_\frac{3}{2}(y)}{I_{-\frac{1}{2}}(y)} - \left( \frac{I_\frac{1}{2}(y)}{I_{-\frac{1}{2}}(y)}
\right)^2 \right]. \label{eq:varexpressionI}
\end{equation}
Because $I_\frac{1}{2}(y)/I_{-\frac{1}{2}}(y)$ can be interpreted as the mean of the density proportional to $z \to z^{-1/2}e^{y^2z/(2\sigma^2)}/(\tau^2 + (1-\tau^2)z)$, and $I_\frac{3}{2}(y)/I_{-\frac{1}{2}}(y)$ as the second moment, it follows that the term in square brackets in \eqref{eq:varexpressionI} is nonnegative. By \eqref{eq:varexpressionI}, we write:
\begin{align} \label{eq:mismatchvarintegrals}
\E_{\theta = \gamma \zt} \var(\theta \mid Y) &=\sigma \int \frac{ I_\frac{1}{2}(\zt + y)}{I_{-\frac{1}{2}}(\zt+y)} \phi\left(\frac{y + (1-\gamma)\zt}{\sigma}\right) dy \notag\\
&\quad +\frac{1}{\sigma}\left( \int_{-\infty}^{-\frac{\zt}{2}} + \int_{-\frac{\zt}{2}}^{3\zt} + \int_{3\zt}^\infty\right) (\zt + y)^2\notag \\
&\quad \cdot\left[
\frac{I_\frac{3}{2}(\zt + y)}{I_{-\frac{1}{2}}(\zt+y)} - \left(
\frac{I_\frac{1}{2}(\zt+y)}{I_{-\frac{1}{2}}(\zt + y)}
\right)^2
\right]
\phi\left(\frac{y + (1-\gamma)\zt}{\sigma}\right) dy.
\end{align} 
The first term in \eqref{eq:mismatchvarintegrals} is as \eqref{eq:mismatchbiasintegrals}, except without the factor $(\zt + y)$. Following the same steps as the proof of \eqref{eq:mismatchbiasgoal}, we see that it is smaller than a multiple of $\zt^{-1}$ times the bound on \eqref{eq:mismatchbiasintegrals}, so it is of the order $\zt^{2\gamma-3}\tau^{(1-\gamma)^2}$. The first and third integrals of the second term of \eqref{eq:mismatchvarintegrals} are also negligible. For the first, we use that the expression in square brackets is nonnegative and bounded above by $I_\frac{3}{2}(y)/I_{-\frac{1}{2}}(y)$, which in turn is bounded above by $I_\frac{1}{2}(y)/I_{-\frac{1}{2}}(y)$. We bound as in \eqref{eq:mismatchbiasint1}, with the difference that the leading factor is $(\zt+y)^2$ instead of $(\zt+y)$. This leads to the order $\zt\tau^{(1-\gamma)^2+\gamma}$, much smaller than the claimed rate. For the third integral, we can bound the term in square brackets by 1 and use Mills' ratio to see that it is of the order $\zt\tau^{(4-\gamma)^2}$.

We are left with the middle integral of the second term of \eqref{eq:mismatchvarintegrals}. On the domain of this integral, by Lemma \ref{lem:mismatch2}:
\begin{equation*}
\frac{I_\frac{3}{2}(\zt+y)}{I_{-\frac{1}{2}}(\zt+y)} = \frac{c_3\tau^4(\zt+y)^2 + 2\sigma^2 e^\frac{y^2+2y\zt}{2\sigma^2}}{c_1(\zt+y)^2 + 2\sigma^2e^\frac{y^2+2y\zt}{2\sigma^2}}(1+o(1)),
\end{equation*}
where $c_3 = \int_0^1 z^{3/2}(1+z)^{-1}dz$, and $c_1$ is as in \eqref{eq:biassecond}. We see that $I_\frac{3}{2}(y)/I_{-\frac{1}{2}}(y)$ and $I_\frac{1}{2}(y)/I_{-\frac{1}{2}}(y)$ are asymptotic to the same function on this domain. Since $A/(A+B) - A^2/(A+B)^2 = AB/(A+B)^2$, it follows that up to $\mathcal{O}(\tau)$, the middle integral is asymptotic to
\begin{align*}
\frac{1}{\sigma} \int_{-\frac{\zt}{2}}^{3\zt} &(\zt+y)^2 \frac{c_1(\zt+y)^22\sigma^2 e^\frac{y^2+2y\zt}{2\sigma^2}}{\left(c_1(\zt+y)^2 + 2\sigma^2e^\frac{y^2+2y\zt}{2\sigma^2} \right)^2} 
\phi\left( \frac{y+(1-\gamma)\zt}{\sigma}\right) dy\\
&= \frac{2\sigma c_1}{\sqrt{2\pi}} \tau^{(1-\gamma)^2} \int_{-\frac{\zt}{2}}^{3\zt} \frac{(\zt+y)^4e^\frac{\gamma\zt y}{\sigma^2}}{
\left(c_1(\zt+y)^2 + 2\sigma^2e^\frac{y^2+2\zt y}{2\sigma^2}\right)^2
}dy.
\end{align*}
We substitute $u = \zt y - 2\sigma^2\log\zt$ to reduce this to 
\begin{align*}
\frac{2\sigma c_1}{\sqrt{2\pi}} \tau^{(1-\gamma)^2}&\frac{1}{\zt} \int_{-\frac{\zt^2}{2}-2\sigma^2\log\zt}^{3\zt^2-2\sigma^2\log\zt} \frac{
(\zt + \zt^{-1}(u + 2\sigma^2\log\zt))^4e^\frac{\gamma u}{\sigma^2}\zt^{2\gamma}
}
{\left(
c_1(\zt + \zt^{-1}(u + 2\sigma^2\log\zt))^2 + 2\sigma^2 e^\frac{u}{\sigma^2}\zt^2
\right)^2} du\\
&\sim \frac{2\sigma c_1}{\sqrt{2\pi}} \tau^{(1-\gamma)^2}\frac{1}{\zt} \int_{-\infty}^\infty \frac{\zt^4 e^\frac{\gamma u}{\sigma^2}\zt^{2\gamma}}
{\left(
c_1\zt^2 + 2\sigma^2\zt^2e^\frac{u}{\sigma^2}
\right)^2} du.
\end{align*}
This is asymptotic to  expression \eqref{eq:mismatchvargoal}, with $c_2(\gamma) = (2\sigma c_1/\sqrt{2\pi})\int_{-\infty}^\infty e^\frac{\gamma u}{\sigma^2}/(c_1 + 2\sigma^2e^\frac{u}{\sigma^2})^2du$.

\end{proof}

\textbf{Proof of Theorem \ref{thm:empiricalbayes}}
\begin{proof}
Suppose that $Y \sim \mathcal{N}(\theta, \sigma^2I_n)$, $\theta \in \ell_0[p_n]$. We adapt the approach in paragraph 6.2 in \citep{Johnstone2004}. We first derive the following inequality for events $A$ such that $\widehat\tau > \tau$ holds with probability one on $A$:
\begin{align}
\E_\theta(T_{\widehat\tau}(Y_i) - \theta_i)^2\1_A 
&\leq 2\E_\theta(T_{\widehat\tau}(Y_i) - Y_i)^2 \1_A+ 2\E_\theta(Y_i - \theta_i)^2\1_A \notag\\
&\lesssim 2\E_\theta\zeta_{\widehat\tau}^2\1_A + 2\sigma^2 \E_\theta Z^2\1_A\label{eq:bound70a}
\end{align}
where (\ref{eq:maxdiff}) was used in the second line, and $Z$ follows a standard normal distribution. If $A$ is such that $\widehat\tau > \tau$ holds with probability one on $A$, we can use the inequality $\zeta_{\widehat\tau} < \zt$ if $\widehat\tau > \tau$ to find:
\begin{align}
\E_\theta(T_{\widehat\tau}(Y_i) - \theta_i)^2\1_A  &\lesssim
  2\zt^2\mathbb{P}_\theta(A) +  2\sigma^2 \E_\theta Z^2\1_A\label{eq:bound70},
\end{align}

  We now consider the nonzero and zero parameters separately. For both cases, we split up the expected $\ell_2$ loss as follows: 
\begin{equation*}
\E_{\theta}(T_{\widehat\tau}(Y_i) - \theta_i)^2= 
\E_{\theta}(T_{\widehat\tau}(Y_i) - \theta_i)^2\1_{\{\widehat\tau > c\tau \}} + \E_{\theta}(T_{\widehat\tau}(Y_i) - \theta_i)^2\1_{\{\widehat\tau \leq c\tau \}},
\end{equation*}
and then bound each of terms on the right hand side. For the nonzero means, we take $c = 1$, while for the zero means, we consider $c \geq 1$. Note that for $\zeta_{\widehat\tau}$ to be well-defined, we need $\widehat\tau \leq 1$ and consequently, when we consider $\widehat\tau > c\tau$, we must have $c\tau < 1$. 

\textit{Nonzero means}\\
By (\ref{eq:bound70}), we find:
\begin{equation}
\E_\theta (T_{\widehat\tau}(Y_i) - \theta_i)^2\1_{\{\widehat\tau > \tau \}} 
\lesssim 2\zt^2 + 2\sigma^2. \label{eq:lemma5a}
\end{equation}
If $\widehat\tau \leq \tau$, the inequality $\zeta_{\widehat\tau}^2 \leq \zt^2$ needed for (\ref{eq:bound70}) does not hold. For this case, we assume that $\widehat\tau \geq g(n, p_n)$ with probability one, for some function $g(n, p_n)$, corresponding to $\zeta_{\widehat\tau} \leq \sqrt{-2\sigma^2 \log g(n, p_n)}$. Then we find by (\ref{eq:bound70a}):
\begin{align}
\E_\theta (T_{\widehat\tau}(Y_i) - \theta_i)^2\1_{\{\widehat\tau \leq \tau \}}
&\lesssim 2\E_\theta \zeta_{\widehat\tau}^2\1_{\{\widehat\tau \leq \tau \}} + 2\sigma^2  \notag\\
&\leq -4\sigma^2\log (g(n, p_n))\mathbb{P}_\theta(\widehat\tau \leq \tau) + 2\sigma^2 \label{eq:lemma6}.
\end{align}

By (\ref{eq:lemma5a}) and (\ref{eq:lemma6}), we have for $\theta_i \neq 0$:
\begin{align}
\E_{\theta}(T_{\widehat\tau}(Y_i) - \theta_i)^2 
&\lesssim 1 + \zt^2  - \log(g(n, p_n)) \mathbb{P}_\theta(\widehat\tau \leq \tau). \label{eq:empiricalbayesnonzero}
\end{align}
\textit{Zero means}\\
We first establish an inequality for $\E_\theta[ Z^2\1_{A}]$, where $A$ is an event and $Z$ a standard normal random variable. By Young's inequality, we have for any positive $x$ and $y$:
\begin{align*}
xy \leq \int_0^x (e^s - 1) ds + \int_0^y \log(s+1) ds = e^x - x - 1 + (y+1)\log(y+1) - y.
\end{align*}
By this inequality combined with the inequality $\log(y+1) < y$,  we have:
\begin{align*}
\E_\theta Z^2\1_A
&\leq cd \E_\theta\left[ e^\frac{Z^2}{c} - \frac{Z^2}{c} - 1  \right] + cd\mathbb{P}_\theta(A)\left( \frac{1}{d}\log\left(\frac{1}{d}+1\right) - \frac{1}{d}\right).
\end{align*}
With $c = 3$ and $d = \mathbb{P}_\theta(A)$, we find:
\begin{align}\label{eq:resultyoung}
\E_\theta Z^2\1_A &\leq (3\sqrt{3} - 4)\mathbb{P}_\theta(A) + 3\mathbb{P}_\theta(A)\log\left(1 + \frac{1}{\mathbb{P}_\theta(A)} \right)\notag \\
&< 5\mathbb{P}_\theta(A)\log\left(1 + \frac{1}{\mathbb{P}_\theta(A)} \right).
\end{align}

By (\ref{eq:bound70}) and (\ref{eq:resultyoung}), we get for any $c \geq 1$ such that $c\tau < 1$:
\begin{align}
\E_\theta (T_{\widehat\tau}(Y_i) - \theta_i)^2\1_{\{\widehat\tau > c\tau \}} 
&\lesssim 2\zt^2 \mathbb{P}_\theta\left( \widehat\tau > c\tau\right) \notag\\
&\quad + 10\sigma^2 \mathbb{P}_\theta\left(\widehat\tau > c\tau\right)\log\left(1 + \frac{1}{\mathbb{P}_\theta\left(\widehat\tau > c\tau\right)} \right). \label{eq:lemma5b}
\end{align}

Now suppose $\widehat\tau \leq c\tau$ for some $c \geq 1$ such that $c\tau < 1$. First note that $|T_\tau(y)|$ increases monotonically in $\tau$, as is clear from 
\begin{equation*}
T_{\tau}(y_i) = \E[(1-\kappa_i)y_i \mid y_i, \tau] = \E\left[\left.\frac{\tau^2\lambda_i^2}{1 + \tau^2\lambda_i^2}y_i \ \right| \ y_i , \tau\right].
\end{equation*}
Because sign$(T_{\widehat\tau}(y_i))$ = sign$(T_{c\tau}(y_i))$ and $0 \leq |T_{\widehat\tau}(y_i)| \leq |T_{c\tau}(y_i)|$, we have:
\begin{equation*}
\left(T_{\widehat\tau}(y_i) - \theta_i\right)^2 \leq \max\{\theta_i^2, (T_{c\tau}(y_i) - \theta_i)^2\} \leq \theta_i^2 + (T_{c\tau}(y_i) - \theta_i)^2.
\end{equation*}
Hence:
\begin{equation*}
\E_\theta (T_{\widehat\tau}(Y_i) - \theta_i)^2\1_{\{\widehat\tau \leq c\tau \}} \leq \theta_i^2 + \E_\theta(T_{c\tau}(Y_i) - \theta_i)^2 .
\end{equation*}
And thus, by  (\ref{eq:goalmsezero}), we have for $\theta_i = 0$:
\begin{equation}\label{eq:lemma5c}
\E_\theta (T_{\widehat\tau}(Y_i) - \theta_i)^2\1_{\{\widehat\tau \leq c\tau \}}  \lesssim \zeta_{c\tau}c\tau \lesssim \zt\tau.
\end{equation}
Combining (\ref{eq:lemma5b}) and (\ref{eq:lemma5c}), we find:
\begin{align}
\E_{\theta}T_{\widehat\tau}(Y_i) ^2
&\lesssim \zt\tau + \zt^2 \mathbb{P}_\theta\left( \widehat\tau > c\tau\right) + \mathbb{P}_\theta\left(\widehat\tau > c\tau\right)\log\left(1 + \frac{1}{\mathbb{P}_\theta\left(\widehat\tau > c\tau\right)} \right).\label{eq:empiricalbayeszero}
\end{align}
\textit{Conclusion}\\
We can now bound the expected $\ell_2$ loss. We assume that $\theta_i \neq 0$ for $i = 1, \ldots,\tilde p_n$ and $\theta_i = 0$ for $i = \tilde p_n + 1, \ldots, n$, where $\tilde p_n \leq p_n$. By combining (\ref{eq:empiricalbayesnonzero}) and (\ref{eq:empiricalbayeszero}), we find:
\begin{align}
\E_\theta \|T_{\widehat \tau}(Y) - \theta\|^2 &\lesssim \tilde p_n\left(1 + \zt^2 - \log(g(n,p_n))\mathbb{P}_\theta(\widehat\tau \leq \tau)\right) + (n-\tilde p_n)\zt\tau \notag \\
&\quad + (n - \tilde p_n) \mathbb{P}_\theta\left(\widehat\tau > c\tau\right)\left(\zt^2 + \log\left(1 + \frac{1}{\mathbb{P}_\theta\left(\widehat\tau > c\tau\right)} \right)\right). \label{eq:empiricalbayesfinal}
\end{align}
The function $x\log\left(1 + \frac{1}{x}\right)$ is monotonically increasing in $x$ for $x \in [0,1]$. Hence, with the choice $\tau = \frac{p_n}{n}$ or $\tau = \frac{p_n}{n}\sqrt{\log(n/p_n)}$, the conditions stated in the theorem are sufficient for (\ref{eq:empiricalbayesfinal}) to be bounded by the minimax squared error rate in the worst case. 

If an estimator $\widehat\tau$ satisfies only the first condition, then $\sup\{\frac{1}{n}, \widehat\tau\}$ satisfies the second condition with $-\log g(n, p_n)  = \log n$. By the assumption $p_n \to \infty$, we have $\mathbb{P}_\theta(\sup\{\frac{1}{n}, \widehat\tau\} > c\frac{p_n}{n}) \leq \mathbb{P}_\theta(\widehat\tau > c\frac{p_n}{n}) $. Plugging this into inequality (\ref{eq:empiricalbayesfinal}) yields an $\ell_2$ risk of at most order $p_n\log n$. 
\end{proof}

\begin{lem}\label{lem:schatter}
Suppose $Y_i \sim \mathcal{N}(\theta_i, \sigma^2), i = 1, \ldots, p_n$ and $Y_i \sim \mathcal{N}(0, \sigma^2), i = p_n+1, \ldots, n$ and define 
\begin{equation*}
\widehat\tau =  \frac{\#\{|y_i| \geq \sqrt{c_1\sigma^2\log n}, i = 1, \ldots, n \}}{c_2n}
\end{equation*}
 for some $c_2 > 1$. Then $\mathbb{P}_\theta\left(\widehat\tau > \tau \right) \lesssim \frac{p_n}{n}$ as $p_n, n \to \infty$, $p_n = o(n)$ if $c_1 > 2$, or $c_1 = 2$ and $p_n \lesssim \log n$ for $\tau = \frac{p_n}{n}$ or $\tau = \frac{p_n}{n}\sqrt{\log(n/p_n)}$. 
\end{lem}

\begin{proof}
We only need to consider $\mathbb{P}_\theta(\widehat\tau > \frac{p_n}{n})$, as we assume $p_n = o(n)$ and thus, for large $n$, $\mathbb{P}_\theta(\widehat\tau > \frac{p_n}{n}\sqrt{\log(n/p_n)}) \leq \mathbb{P}_\theta(\widehat\tau > \frac{p_n}{n})$.
Define $A_i = \{|y_i| \geq \sqrt{c_1\sigma^2\log n}\}, i = 1, \ldots, n$. For $i = p_n +1, \ldots, n$, $\1_{A_i}$ follows a Bernoulli distribution with parameter $q_n = 2\Phi^c(\sqrt{c_1\log n})$, which by Mills' ratio can be bounded from above by $\sqrt{\frac{2}{c_1\pi}} \left(\log n\right)^{-\frac{1}{2}} n^{-\frac{c_1}{2}}$. For $X \sim \text{Bin}(n, p)$, we have the bound $\mathbb{P}(X \geq k) \leq \left(\frac{enp}{k}\right)^k$ as a consequence of Theorem 1 in \citep{Chernoff1952}. Hence:
\begin{align}
\mathbb{P}_\theta\left(\widehat\tau > \frac{p_n}{n}\right) 
&\leq \mathbb{P}_\theta \left(\sum_{i=p_n+1}^n \1_{A_i} > (c_2-1)p_n  \right) 
\leq \left(\frac{e(n-p_n)q_n}{(c_2-1)p_n+1}\right)^{(c_2-1)p_n + 1}\notag\\
&\leq \left(\sqrt{\frac{2e^2}{c_1\pi}} \frac{1}{(c_2-1)p_n + 1} \frac{1}{\sqrt{\log n}} n^{1-\frac{c_1}{2}}\right)^{(c_2-1)p_n + 1}.\label{eq:chernoff}
\end{align}
The inequality   $\mathbb{P}_\theta\left(\widehat\tau > \frac{p_n}{n}\right) \lesssim \frac{p_n}{n}$ holds if $-\log \mathbb{P}_\theta\left(\widehat\tau > \frac{p_n}{n}\right) \geq \log \frac{n}{p_n} + c$ holds for some positive constant $c$. The negative logarithm of bound \eqref{eq:chernoff} is:
\begin{equation*}
((c_2-1)p_n+1)\left(\frac{1}{2}\log \frac{c_1\pi}{2e^2} + \log((c_2-1)p_n+1) + \frac{1}{2}\log\log n + \left(\frac{c_1}{2} - 1 \right)\log n\right).
\end{equation*}
For $c_1 = 2$, this quantity will exceed $\log\frac{n}{p_n}$ if $p_n \gtrsim \log n$. If $c_1 > 2$, we require $((c_2-1)p_n + 1)(\tfrac{c_1}{2}-1) \geq 1$, which is certainly satisfied if $p_n \to \infty$.

\end{proof}

\vspace*{-0.5cm}
\bibliographystyle{imsart-nameyear}
\bibliography{horseshoe}

\end{document}